\documentclass[11pt,twoside]{amsbook}
\usepackage[all]{xy}   
        \usepackage {amssymb,latexsym,amsthm,amsmath}
        \usepackage{enumitem, tikz-cd}
        
\long\def\symbolfootnote[#1]#2{\begingroup%
\def\thefootnote{\fnsymbol{footnote}}\footnote[#1]{#2}\endgroup}

\newcommand{\tr}{\ensuremath{{}^t\!}}
\newcommand{\tra}{\ensuremath{{}^t}}

\newcommand{\x}{\mathfrak X}

\newcommand{\Aut}{\textup{Aut}}
\newcommand{\diag}{\textup{diag}}
\newcommand{\End}{\textup{End}}
\newcommand{\cross}{\times}

\makeatletter
\def\imod#1{\allowbreak\mkern10mu({\operator@font mod}\,\,#1)}
\makeatother

\numberwithin{section}{chapter}
\newtheorem{theorem}{Theorem}[section]
\newtheorem{lemma}[theorem]{Lemma}

\newtheorem{proposition}[theorem]{Proposition}

\newtheorem{definition}[theorem]{Definition}
\newtheorem*{theorem*}{Theorem}
\theoremstyle{definition}
\newtheorem{exercise}[theorem]{\bf Exercise}

\numberwithin{equation}{chapter}

\newcommand{\ignore}[1]{}

\newcommand{\mynote}[1]{}
\begin{document}
\setcounter{section}{0}
\title{From Linear Algebra to Algebraic Groups  \\ {\large An exercise based approach to matrix groups}}
\author{Anupam Singh}
\address{IISER Pune, Dr. Homi Bhabha Road, Pashan, Pune 411 008 INDIA}
\email{anupamk18@gmail.com}
\date{}


\maketitle
\tableofcontents

\chapter{Introduction}

What's common among the great batsman Sachin Tendulkar, the great athlete Usion Bolt, the great tennis player Roger Fedrer and the great footballer Maradonna?

{\bf Answer : They spend hours after hours, days after days training!} 

Over the years, the author has taught material presented here to the graduate and undergraduate students at IISER Pune and in various summer and winter schools meant for graduate students and/or teachers. The author would like to thank the participants of various ATM schools who enthusiastically responded to the lectures and tutorials.  
\begin{enumerate}
\item {\bf Why this?} The idea behind these notes is to introduce matrix groups. The groups we mostly encounter are matrix groups: most of the finite simple groups, algebraic groups, Lie groups etc. Yet, in our undergraduate course on group theory we do not introduce matrix groups. This is an attempt to correct that mistake.
\item {\bf For whom?} This book grew out of an attempt to convince audience that they can make use of Linear Algebra they know, to study some of the matrix groups. Any PhD student, undergraduate student or teacher will benefit from this. The topics presented here form bread and butter for group theorists, number theorists, physicists and engineers.  
\item {\bf How to read this?}
This note is meant to be a supplement for those who have studied first course in Linear Algebra and Group Theory. That is why its in the form of exercises. The author confesses that some of the exercises throw big, unknown words but they are not meant to scare you. It's hoped that this will enthuse you to look further.
\end{enumerate}
We present some motivation behind this attempt. Author hopes that these stories will enthuse the audience towards this material.

\section{CFSG}
One of the biggest achievements of the last century is ``The Classification of Finite Simple Groups''. It was born out of the following: 

{\bf Can we classify all finite groups?}

Let $G$ be a finite group. If $G$ has a proper normal subgroup $N$, it can be broken into two pieces: $N, G/N$. However, usually it is not easy to construct the group $G$ given the smaller groups. They fit in an exact sequence:
$$1\rightarrow N \rightarrow G \rightarrow G/N \rightarrow 1.$$
The problem of constructing $G$ out of given other two is called "Extension Problem" and is a difficult problem. For example, to determine central extension one uses some cohomology theory. 

However, using the above idea we reduce the problem of classifying all finite groups to the classification of finite simple groups. This has been achieved now. It has taken efforts of more than a hundred year and several hundred mathematicians. Thus it is usually not attributed to any one of them.
\begin{theorem}[CFSG]
A finite simple group is one of the following:
\begin{enumerate}
\item Cyclic groups $\mathbb Z/p\mathbb Z$, where $p$ is a prime.
\item Alternating groups $A_n$ for $n\geq 5$.
\item Finite groups of Lie type
\begin{enumerate}
\item classical type 
\item exceptional type
\end{enumerate} 
\item The sporadic groups $26$ of them.
\end{enumerate}
\end{theorem}
In our undergraduate group theory course we learn the first two family. The idea here is to attempt to convince the audience that they can get familiar with the third family as well using their knowledge of Linear Algebra, at least for sure the group of classical type. For the history and more information please read the wikipedia article. Here are the classical groups:
\begin{enumerate}
\item Special Linear Groups
\item Orthogonal Groups of odd dimension
\item Symplectic Groups
\item Orthogonal Groups of even dimension
\item Unitary Groups
\end{enumerate}

\section{Representation theory}
Let $G$ be a group. While working with some examples of groups we quickly realise that there is not much we know about them. One way is possibly to compare them with some known groups. In this case we take the known groups as linear group, $GL(V)$, and the comparison is made via all group homomorphisms from $G$ to linear groups. More precisely, we define a representation.

Fix a field $k$. A representation of $G$ is a group homomorphism $\rho \colon G \rightarrow GL(V)$ where $V$ is a vector space over the field $k$. A subspace $W\subset V$ is said to be invariant if $\rho(g)(W)\subset W$, for all $g\in G$. Further, a representation is said to be irreducible if it has no proper invariant subspace.

For a finite group $G$, it turns out that over algebraically closed field $k$ when $char(k)\nmid |G|$ the representations of the group can be understood via character theory. In this case every representation can be written as a direct sum of irreducible representations. Thus, the character theory is built on knowing basic group theory and linear algebra well.

\section{Computational group theory}
To understand a subject it is often important to know many examples. Even better, if one could compute with these examples and verify before making any conjecture or trying to solve a problem, would be divine.
Due to advancement in computational power the idea is to implement groups in the computer so that one is able to compute with those examples. There are several computational softwares SAGE, GAP and MAGMA. To make the idea clear: like any software there are two group of people, developers and users. It goes without saying that knowing these packages adds an extra dimension to one's own research. It's interesting to learn some of the fundamental problems in CGT.

\section{Robotics, Computer graphics etc}
In robotics basic mathematics suc as rotation and translation in 3 dimension is used. One often works with the groups $O_2(\mathbb R)$ and $O_3(\mathbb R)$, and slightly more general, isometry groups. In computer graphics Hamilton's quaternions are used to achieve the desired results.


\chapter{The Group $SL_2$}
The group $SL_2$ is ubiquitous in mathematics. It appears 
\begin{enumerate}
\item in algebraic groups as a smallest example of simple algebraic groups $SL_2$,
\item in Lie groups as the group $SL_2(\mathbb R)$,
\item in complex analysis as mobius transformations $SL_2(\mathbb Z)$,
\item in topology as an example of free product $PSL_2(\mathbb Z)\cong \mathbb Z/2\mathbb Z * \mathbb Z/3\mathbb Z$,
\item in finite group theory as a part of the finite simple groups $PSL_2(q)$,
\item in number theory as the action of $PSL_2(\mathbb Z)$ on the upper half plane, 
\item in K-theory as the group $SL_2(R)$ over a ring $R$ with identity etc.
\end{enumerate}
The exercises below are devoted to the ubiquity of this group.

\section{Generation and Automorphisms of the group $SL_2(k)$}
Let $k$ be a field. The group $SL_2(k)=\left\{X\in M_2(k)\mid \det(X)=1 \right\}$. Now define $x_{12}(t):= \begin{pmatrix} 1&t\\ &1\end{pmatrix}$ and $x_{21}(s):= \begin{pmatrix} 1&\\s &1\end{pmatrix}$ for $t, s\in k$. The matrices $x_{12}(t)$ and $x_{21}(s)$ are called {\bf elementary matrices}.

\begin{exercise} With the notation as above, 
\begin{enumerate}
\item The maps $\mathfrak{X}_{12}\colon k \rightarrow SL_2(k)$ given by $t\mapsto x_{12}(t)$ and $\mathfrak{X}_{21}\colon k \rightarrow SL_2(k)$ given by $t\mapsto x_{21}(t)$ are group homomorphisms where $k$ is additive group of the field.
\item Compute $w(t)=x_{12}(t)x_{21}(-t^{-1})x_{12}(t)$. What is $w(-1)$?
\item Compute $h(t)=w(t)w(-1)$.
\item For $c\neq 0$, verify the following for an element in $SL_2(k)$, 
$$\begin{pmatrix}a&b\\c&d\end{pmatrix}= \begin{pmatrix}1&(a-1)c^{-1}\\&1\end{pmatrix} \begin{pmatrix}1&\\c&1\end{pmatrix} \begin{pmatrix}1&(d-1)c^{-1}\\ &1\end{pmatrix}.$$
\item Show that the group $SL_2(k)$ is generated by the set of all elementary matrices $\{x_{12}(t), x_{21}(s) \mid t,s\in k\}$. 
\end{enumerate}
\end{exercise}
\noindent This exercise shows that the elementary matrices generate the group. These generators are also called {\bf Chevalley generator}.

\begin{exercise}[Bruhat decomposition]
Let $B=\left\{\begin{pmatrix}a&b\\&a^{-1}\end{pmatrix}\right\}$ be the set of upper triangular matrices in $SL_2(k)$. Show the following double coset decomposition of $SL_2(k)$,
$$B\backslash SL_2(k)/B = B\bigcup B \begin{pmatrix}&-1\\1&\end{pmatrix} B.$$
\end{exercise}
\noindent To solve this exercise take a matrix which is not upper triangular. Now, by multiplying appropriate $x_{12}(t)$ from, first left and then right, reduce it to of the form $w(s)$. Among other things, this exercise shows that the group is generated by mostly upper triangular matrices provided we throw in the permutation matrix. However, as $t$ is varying over $k$ it seems we need a large number of generating elements. The following exercise says that that is not necessarily the case.

\begin{exercise}
Prove that the group $SL_2(\mathbb F_p)$ is generated by the two of its elements $\begin{pmatrix} 1& 1 \\ &1\end{pmatrix}$ and $\begin{pmatrix} 1&  \\1 &1\end{pmatrix}$.
\end{exercise}
\noindent 

\begin{exercise}
Prove that the commutator subgroup $[SL_2(k), SL_2(k)]=SL_2(k)$ if $|k|\geq 5$.
\end{exercise}
\noindent In the view of earlier exercises it is enough to show that the elementary matrices are commutator.

\begin{exercise} Let $k$ be a field with $|k| \geq 5$.
Show that any group homomorphism $\rho \colon SL_2(k) \rightarrow k^*$ is trivial.
\end{exercise}

\begin{exercise}[Automorphism]
\begin{enumerate}
\item Let $g\in GL_2(k)$. The map $\iota_g \colon SL_2(k) \rightarrow SL_2(k)$ defined by $x\mapsto gxg^{-1}$ is a group automorphism.
\item Show that this gives a group homomorphism $GL_2(k) \rightarrow \Aut(SL_2(k))$ defined by $g\mapsto \iota_g$. What is the kernel of this map?
\item Show that the map $d\colon A\mapsto \tr A^{-1}$ is an automorphism of $SL_2(k)$.
\item Let $\sigma\in \Aut(k)$ be a field automorphism of $k$. Define a map $\sigma \colon SL_2(k) \rightarrow SL_2(k)$ by $\begin{pmatrix} a&b\\c&d\end{pmatrix} \mapsto \begin{pmatrix} \sigma(a)&\sigma(b)\\\sigma(c)&\sigma(d)\end{pmatrix}$. Show that it is an automorphism.
\end{enumerate} 
\end{exercise}
\noindent 
This exercise describes automorphisms of the group $SL_2(k)$. The first kind of automorphisms, $\iota_g$, are called inner automorphism if $g\in SL_2(k)$ and diagonal automorphism otherwise. The automorphism $d$ given by transpose-inverse is called the graph automorphism. This comes from the Dynkin-diagram automorphism. The maps induced by field automorphisms are called field automorphism. It is a theorem of Steinberg and further extended by Humphreys that any abstract automorphism of the group $SL_2(k)$ is of the form $\iota_g d \sigma$, that is, compositions of the above maps. Further, if we consider only algebraic group automorphisms the field automorphisms are not there, as they are not morphism of varieties. 
 
\section{The group $SL_2(\mathbb Z)$} 
The group $SL_2(\mathbb Z)=\left\{X\in M_2(\mathbb Z)\mid \det(X)=1 \right\}$.
Let us recall some transformations from complex analysis. We denote by $\mathbb C^+$, the extended complex plane.
\begin{exercise}
Consider the upper half plane $U=\{z\in \mathbb C \mid Im(z)>0\}\cup \{\infty\}$. Denote $U_1=\{z\in U\mid |z|<1\}$.
\begin{enumerate}
\item[(a)] Show that a M\"{o}bius Transformation $T\colon \mathbb C^+ \rightarrow \mathbb C^+$ defined by $T(z)=\frac{az+b}{cz+d}$ where $a,b,c,d\in \mathbb R$ with the property $ad-bc=1$ maps $U$ to $U$. Thus, the group of M\"{o}bius Transformations, $SL_2(\mathbb Z)$ as well as $SL_2(\mathbb R)$, act on $U$.
\item[(b)] Consider $s\colon U\rightarrow U$ given by $s(z)=-\frac{1}{z}$ and $\tau\colon U\rightarrow U$ given by $\tau(z)=z+1$. Find the image of an element $z\in U_1$ under $s,\tau,s\tau, \tau s, s\tau s$ and $\tau s\tau$. Draw pictures in each case as what these elements do to $U_1$ and $\mathbb C\backslash U_1$.
\item[(c)] Show that $-1$ acts trivially.
\item[(d)] Show that $s^2=1, (s\tau)^3=-1$ and for any integer 
$n>0$,  $(s\tau s)^n\neq 1, (\tau s\tau)^n\neq 1$.
\end{enumerate}
\end{exercise}

\begin{exercise} Using the exercise above solve the following.
\begin{enumerate}
\item Show that the group $SL_2(\mathbb Z)$ is generated by the two matrices $S:=\begin{pmatrix} & -1 \\1&\end{pmatrix}$ and $T:=\begin{pmatrix} & -1\\ 1& 1\end{pmatrix}$. Show that the subgroup generated by $<S,T> = <S,U>$ where $U=\begin{pmatrix} 1& 1\\ & 1\end{pmatrix}$.
\item Use this to prove that the group $PSL_2(\mathbb Z)$ is a free product of $\mathbb Z/2\mathbb Z$ and $\mathbb Z/3\mathbb Z$, i.e., $PSL_2(\mathbb Z)\cong \mathbb Z/2\mathbb Z*\mathbb Z/3\mathbb Z$.
\item Use this to prove that the group $SL_2(\mathbb Z)$ is an amalgamated free product of $\mathbb Z/4\mathbb Z$ and $\mathbb Z/6\mathbb Z$ over $\mathbb Z/2\mathbb Z$, i.e., $SL_2(\mathbb Z)\cong \mathbb Z/4\mathbb Z*_{\mathbb Z/2\mathbb Z}\mathbb Z/3\mathbb Z$.
\end{enumerate}  
\end{exercise}
\noindent Here is a brief idea how to go about it. The elements, $\pm I$, are obtained by computing power of $S$. Then show that the elements of the form $\begin{pmatrix} 1& 1\\ & 1\end{pmatrix}$ can be obtained by computing $S^{-1}T$. Use this to prove that $\begin{pmatrix} 1& t\\ & 1\end{pmatrix}$ matrices can be generated by taking power (including inverse) of the previous one. By multiplying with $-I$ show that any upper triangular matrix can be produced. Now to get a general matrix $\begin{pmatrix} a& b\\ c&d\end{pmatrix}$ we may assume $|a|\geq |c|$. Use division algorithm and multiplication by the upper triangular matrices (from right) to reduce the size of $a$. 

\begin{exercise}
Finite groups are algebraic groups, thus $<S>$ and $<T>$ are algebraic subgroups of $SL_2(\mathbb C)$. However, the subgroup $<S,T>=SL_2(\mathbb Z)$ is not an algebraic subgroup. 
\end{exercise}
Note that this exercise also points out that subgroup generated by two finite order elements could be infinite.

\section{Topological group $SL_2(\mathbb R)$}
We consider $M_2(\mathbb R)$ as Euclidean space $\mathbb R^4$ with usual metric topology. 
\begin{exercise}
Show that the map $\det \colon M_2(\mathbb R) \rightarrow \mathbb R$ given by $A\mapsto \det(A)$ is a continuous map.  
\end{exercise}

\begin{exercise}
The set $GL_2(\mathbb R)$ is an open set and $SL_2(\mathbb R)$ is a closed set.
\end{exercise}

\begin{exercise}
\begin{enumerate}
\item Show that $SL_2(\mathbb R)$ is connected. (You can use the idea of Chevalley generators from previous section which gives one-parameter subgroups.)
\item Show that $SL_2(\mathbb R)$ is not compact.
\end{enumerate}
\end{exercise}

\begin{exercise}
Show that $SO_2(\mathbb R)$ is a maximal compact connected subgroup of $SL_2(\mathbb R)$.
\end{exercise}

\begin{exercise}[Iwasawa Decomposition]
Let $A=\begin{pmatrix} a& b \\ c& d\end{pmatrix} \in SL_2(\mathbb R)$. Use Gram-orthonormalisation process on the row vectors of the matrix $A$ to convert it to orthonormal vectors. This idea will prove the following results:
\begin{enumerate}
\item Write $A=PS$ where $P=\begin{pmatrix} \alpha & \beta \\ & \delta \end{pmatrix}$ with $\alpha, \delta > 0$ and $S\in O_2(\mathbb R)$. 
\item This shows that as a topological space $SL_2(\mathbb R) \cong \mathbb R^{+}\times \mathbb R^{+} \times \mathbb R \times O_2(\mathbb R) $.
\item Using this prove that $SL_2(\mathbb R)$ is connected.
\item Prove that the fundamental group $\pi_1(SL_2(\mathbb R))\cong \mathbb Z$.
\end{enumerate}
\end{exercise}

\section{Conjugacy classes and representations}

\begin{exercise}
Determine the conjugacy classes in $GL_2(\mathbb C), GL_2(\mathbb R), SL_2(\mathbb C), SL_2(\mathbb R)$.
\end{exercise}
\begin{exercise}
Show that the unipotents $\begin{pmatrix} 1&a\\&1\end{pmatrix}$ form one conjugacy class in $GL_2(k)$. However, they form $k^*/{k^*}^2$ many conjugacy classes in $SL_2(k)$. Thus when $k=\mathbb Q$, they form infinitely many conjugacy classes of unipotents.
\end{exercise}

\begin{exercise}
Consider the group $SL_2(\mathbb R)$. This group has a natural representation on $2$-dimensional real vector space $V_2$. We write this, say with basis $\{x,y\}$, as follows: for $g\in SL_2(\mathbb R)$ write $g=\begin{pmatrix} a& b \\c&d\end{pmatrix}$ then, $$g.x=ax+cy \ \ {\rm and}\ \ g.y=bx+dy.$$ 
Now consider a real vector space $V_{n+1}$ for $n\geq 1$ with basis 
$$\{x^n, x^{n-1}y, x^{n-2}y^2,\ldots, xy^{n-1}, y^n\}.$$
We define a representation $\rho_{n+1}\colon SL_2(\mathbb R) \rightarrow GL(V_{n+1})$ as follows: 
$$\rho_{n+1}(g)(x^iy^j):=(g.x)^i(g.y)^j$$
where right hand side is expanded as polynomials. Show that $\rho_{n+1}$ is an irreducible representation.
\end{exercise}
\noindent In this exercise one may consider the vector space $V_{n+1}$ as the space of all homogeneous polynomials of degree $n$. Another way to think about this space is that $V_{n+1}\cong Sym^n(V_2)$, the $n^{th}$ symmetric power. To prove irreducibility one makes use of the two nice subgroups, the diagonals and the compact subgroup $SO_2(\mathbb R)$. 

\begin{exercise}
Compute all conjugacy classes of the group $GL_2(\mathbb F_q)$ and $SL_2(\mathbb F_q)$ and number of elements in each conjugacy class. Make a table for the same.
\end{exercise}

Now we introduce more general version of this group. Let $R$ be a commutative ring with identity. Then $SL_2(R)$ is a group. Further the elementary matrices $x_{12}(t)$ and $x_{21}(t)$ are in $SL_2(R)$ for $t\in R$. The subgroup generated by all of the elementary matrices is called the elementary subgroup $E_2(R)$. In general, $E_2(R)$ is a subgroup of $SL_2(R)$. One of the questions is to determine for what $R$ they are equal. Certainly this is the case when $R$ is a field. 

\section{$SL_2(\mathbb Q)$}
Let $K=\mathbb Q(\sqrt d)$ be a degree two field extension of $\mathbb Q$ where $d$ is a squarefree integer. Denote by $K^1=\{x\in K \mid N(x)=1\}$ where $N(a+b\sqrt d)=a^2-b^2d$.
\begin{exercise}
Show that the map $\phi\colon K^* \rightarrow GL_2(\mathbb Q)$ defined by $a+b \sqrt d \mapsto \begin{pmatrix} a & bd \\b& a  \end{pmatrix}$ is an injective group homomorphism.
\end{exercise}

\begin{exercise}
Show that the map $\phi \colon K^1 \rightarrow SL_2(\mathbb Q)$ defined by $a+b \sqrt d \mapsto \begin{pmatrix} a & bd \\b& a  \end{pmatrix}$ is an injective group homomorphism.
\end{exercise}

\begin{exercise}
 Show that all elements of $K^*$ (identified with the image via $\phi$) are diagonalisable in $GL_2(\mathbb C)$.
\end{exercise}

\begin{exercise}
Show that $\phi(K^*)$ is a centraliser of an element. 
\end{exercise}

\begin{exercise}
 Let $K$ and $L$ be two non-isomorphic field extensions of $\mathbb Q$. Show that the image subgroups of $K^*$ and $L^*$ are not conjugate in $GL_2(\mathbb Q)$. 
\end{exercise}

\chapter{$GL_n(k)$ and $SL_n(k)$}

Let $k$ be a field. Let $V$ be a vector space of finite dimension over $k$, say dimension of $V$ is $n$. By fixing a basis $\mathcal B=\{e_1,\ldots, e_n\}$ we can identify $V$ with column vectors $k^n$. We denote the set $GL(V)$ (general linear group) by the set of all invertible linear transformation. We know that using a fixed basis we can represent any element of $GL(V)$ be an invertible matrix, thus we can identify $GL(V)$ with $GL_n(k)$. We wish to study $GL_n(k)$ as a group. The group $SL(V)$ (special linear group) is the linear transformations of determinant $1$ and are written in matrix form as $SL_n(k)$.
The identity matrix is usually written as $I$ and $e_{i,j}$ represents the matrix with $1$ at $ij^{th}$ place and $0$ elsewhere.

\section{Some actions} 
\begin{exercise}
The group $GL(V)$ acts on $V$ by evaluation. Show that there are exactly two orbits $\{0\}$ and $V\backslash \{0\}$. 
\end{exercise}
The set of all one dimensional subspaces of $V$ is written as $\mathbb P(V)$, called projective space. This can be defined in another way as follows. Define an equivalence relation on $V\backslash \{0\}$ by declaring $v$ is related to $w$ if there exists an scalar $\lambda \in k^*$ such that $v=\lambda w$. Then the set $\mathbb P(V)$ is all equivalence classes. 
\begin{exercise}
Show that $GL(V)$ acts on $\mathbb P(V)$ transitively. This action has kernel which is the center of the groups. 
\end{exercise}
\begin{exercise}
Show that a matrix $A\in M_n(k)$ is invertible if and only if its row vectors form a basis of $k^n$ if and only if its column vectors form a basis of $k^n$. 
\end{exercise}
\begin{exercise}
Let $k=\mathbb F_q$ be the finite field with $q$ elements. 
\begin{enumerate}
\item Show that $|GL_n(\mathbb F_q)| = (q^n-1)(q^n-q)(q^n-q^2)\cdots (q^n-q^{n-1})$.
\item Compute the size of $|\mathbb P(\mathbb F_q^{n})|$?
\end{enumerate}
\end{exercise}
\noindent Now we introduce some more groups: projective general linear group 
$$PGL_n(k) = \frac{GL_n(k)}{\mathcal Z(GL_n(k))} $$
and projective special linear group 
$$PSL_n(k) = \frac{SL_n(k)}{\mathcal Z(SL_n(k))}.$$
\begin{exercise}
Show that the center $\mathcal Z(GL_n(k)) =\{\lambda.I_n \mid \lambda \in k^*\} \cong k^*$ and  $\mathcal Z(SL_n(k)) =\{\lambda.I_n \mid \lambda \in k^*, \lambda^n=1\}$.
\end{exercise}
\begin{exercise}
\begin{enumerate}
\item Show that the group $PGL_n(k)$ acts transitively on the set of all $(n+1)$-tuples of points from $\mathbb P(k^n)$.   
\item Show that $PGL_n(k)$ is $2$-transitive on the points of $\mathbb P(k^n)$.
\item Show that $PGL_2(k)$ is $3$-transitive on $\mathbb P(k^1)$.
\end{enumerate}
\end{exercise}
\begin{exercise}
Let $G=GL_n(k)$ and $H=SL_n(k)$. Show that $H$ is a normal subgroup of $G$. Consider the subgroup $K=\{\diag(\lambda, 1\ldots,1)\}$. Show that $G$ is a semi-direct product of $H$ and $K$. Thus, the following short exact sequence is split:
$$1\rightarrow SL_n(k) \rightarrow GL_n(k) \stackrel{\det}{\rightarrow} k^* \rightarrow 1.$$
\end{exercise}

\begin{exercise}
Let $G=GL_n(\mathbb F_p)$ where $p$ is a prime.
\begin{enumerate}
\item What is the $|G|$?
\item Show that the set of upper triangular matrices with all diagonal entries $1$ form a Sylow $p$-subgroup of the group $GL_n(\mathbb F_p)$.
\item Show that $G$ has $(1+p)(1+p+p^2)(1+p+p^2+p^3)\cdots (1+p+\cdots+p^{n-1})$ Sylow $p$-subgroups.
\end{enumerate}
\end{exercise}
For the above exercise, consider $U$ as a Sylow $p$-subgroup consisting of  upper triangular matrices with all diagonals $1$. Consider its normalizer $N_G(U)$. Show that $N_G(U)=B$, the set of all upper triangular matrices. Then the number $n_p=|G|/|N_G(U)|=|G|/|B|$.

\begin{exercise}[Grassmanian]
Let $V$ be a vector space over a field $k$ of dimension $n$. Consider the set of all $r$-dimensional subspaces denoted as $G^r(V)$. 
\begin{enumerate}
\item Show that $GL(V)$ acts on $G^r(V)$ given by $g.W=g(W)$ transitively. 
\item Determine the stabiliser. Can we restrict this action to $SL(V)$?
\item Show that the number of elements in $G^r(\mathbb F_q^n)$ is 
$$\frac{(q^n-1)(q^n-q)\cdots (q^n-q^{r-1})}{(q^r-1)(q^r-q)\cdots (q^{r}-q^{r-1})}.$$ 
\item Show that $|G^r(\mathbb F_q^n)| = |G^{n-r}(\mathbb F_q^n)|$. 
\end{enumerate}
\end{exercise}
\noindent The space $G(V)=\bigoplus G^r(V)$ is called Grassmanian. It is a projective variety.

\begin{exercise}
Let $V$ be a vector space and $V^*$ its dual space. For a subspace $W$ of $V$ define $W^*=\{f\in V^*\mid f(w)=0\forall w\in W\}$. Show that $W\mapsto W^*$ is a inclusion reversing one-one map between subspaces of $V$ and $V^*$.
\end{exercise}

\begin{exercise}[Cayley-Hamilton theorem]
\begin{enumerate}
\item[(a)] Let $A=(a_{ij})$ be a $n\times n$ upper triangular matrix, i.e., $a_{ij}=0$ for all $i>j$. Find its characteristic polynomial $\chi_A(X)$.
\item[(b)] Prove the Cayley-Hamilton theorem for an upper triangular matrix $A$ (by direct calculation), that is, $\chi_A(A)\equiv 0$.
\item[(c)] Use the part (b) to prove Cayley-Hamilton theorem for any complex matrix $S$, i.e., show that $\chi_S(S)\equiv 0$.
\item[(d)] Let $R$ be a commutative ring with $1$. Let $S\in M_n(R)$. Let $\chi_S(x)=\det(x.I-S)$ be the characteristic polynomial over $R$. Show that $\chi_S(S)\equiv 0$. 
\end{enumerate}
\end{exercise}

\section{Bruhat Decomposition}
Denote the subgroup of all upper triangular matrices in $GL_n(k)$ by $B$. It is also called a Borel subgroup. The set of permutation matrices obtained by permuting the rows of identity matrix is a copy of the symmetric group in $GL_n(k)$. Then,
\begin{theorem}
The double coset decomposition
$$B\backslash GL_n(k) / B = \bigsqcup_{w\in S_n} BwB $$
where $S_n$ denotes the set of all permutation matrices. 
\end{theorem}
\noindent Denote by $T$ the set of all diagonal matrices in $GL_n(k)$. It is called a maximal torus.
\begin{exercise}
\begin{enumerate}
\item Prove that the normaliser of $T$ is the set of all monomial matrices. Show that $\mathcal Z_{GL_n(k)}(T) = T$. 
\item Show that $N_{GL_n(k)}(T)/T \cong S_n$.
\end{enumerate}
\end{exercise}

Let us recall the traditional row-column operations. There are three row operations and three column operations each. The row operation $R1$ is to multiply an $i^{th}$ row by some $t\in k$ and add to $j^{th}$ row where $i\neq j$. The row operation $R2$ interchanges an $i^{th}$ row with $j^{th}$ for $i\neq j$. The row operation $R3$ multiplies a row by a scalar $\lambda\in k^*$. Similarly, there are three column operations $C1, C2$ and $C3$ analogous to the row operations.
Let $k$ be a field and for $t\in k$, define $x_{i,j}(t)=I+te_{i,j}\in GL_n(k)$ for $i\neq j$ where $e_{i,j}$ represents the matrix with $1$ at $ij^{th}$ place and $0$ elsewhere. 
\begin{exercise}
\begin{enumerate}
\item Show that the multiplication by $x_{i,j}(t)$ from left (right) to a matrix $A$ is the row (column) operation $R1$ ($C1$).
\item Show that the multiplication by a permutation matrix from left (right) is the row operation $R2$ ($C2$).
\item Show that the multiplication by a diagonal matrix $\diag(1,\ldots, 1, \lambda, 1, \ldots, 1)$ where $\lambda$ is at $i^{th}$ place, from left (right) is the row (column) operation $R3$ ($C3$). 
\item Let $A$ be an invertible matrix. Show that by using the row-column operations we can reduce it to identity. 
\item Show that inverse of row-column operations are themselves. This algorithm is used to compute inverse of a matrix.
\end{enumerate}
\end{exercise}
\noindent However, now we allow only the first row and column operations $R1$ and $C1$. 
\begin{exercise}
By computing $n_{ij}(1)=x_{ij}(1)x_{ji}(-1)x_{ij}(1)$ show that we can still do the $R2$ and $C2$ (well ! almost).
\end{exercise}
\noindent Thus we introduce our new row-column operations $NR2$ and $NC2$ where we multiply by the above matrices $n_{ij}$ replacing $R2$ and $C2$. 
\begin{exercise}
By computing $h_{ij}=n_{ij}(t)n_{ij}(-1)$ where $n_{ij}(t)=x_{ij}(t)x_{ji}(-t^{-1})x_{ij}(t)$ show that $h_{ij}(t)$ is a diagonal matrix with $t$ at $i^{th}$ place and $t^{-1}$ at $j^{th}$ place. Use this to show we can still do $R3$ and $C3$, again almost.
\end{exercise}
\noindent Now, we introduce new row-column operations $NR3$ and $NC3$ where we multiply by the above matrices $h_{ij}(t)$ replacing $R3$ and $C3$. 
\begin{exercise}[Gaussian Elimination]
Now we have row-column operations as $R1, NR2, NR3$ and $C1, NC2, NC3$. 
\begin{enumerate}
\item Show that, using row-column operations, any invertible matrix $A$ can be reduced to a diagonal matrix $\diag(1,\ldots,1,\det(A))$.
\item Show that if we use row-column operations $x_{i,j}(t)$ for $i<j$, that is only half of the $R1$ and $C1$, then any invertible matrix $A$ can be written as $U_1WU_2$ where $U_1$ and $U_2$ are upper triangular matrices and $W$ is a monomial (permutation of a diagonal) matrix.
\item Show that any matrix $A$ (not necessarily invertible) can be reduced to a diagonal matrix of kind $\diag(1,\ldots,1,\det(A))$ or $\diag(1,\ldots,1,0,\ldots,0)$ by row-column operations. Determine the number of non-zeros on the diagonal.
\item Determine when an invertible matrix can be written as product of an upper triangular matrix and a lower triangular matrix?
\item Prove that $SL_n(k)$ is generated by the elementary matrices $x_{ij}(t)$ where $i\neq j$ and $t\in k$.
\end{enumerate}
\end{exercise}

\begin{exercise}
Let $N\in M_n(k)$ be a nilpotent matrix. Show that $\exp(N):=\sum_{r\geq 0} \frac{N^r}{r!}$ has only finitely many terms and is an invertible matrix. 
\end{exercise}
\begin{exercise}
Consider the set $sl_n(k)=\{A\in M_n(k) \mid trace(A)=0\}$. Compute the exponentials of $e_{ij}$ where $i\neq j$. 
\end{exercise}

\section{Parabolic subgroups}
Now we introduce the parabolic subgroups.
Let $V$ be a vector space over field $k$ of dimension $n$. A flag $\mathcal F =(V_0, V_1, V_2, \cdots, V_r)$ is a strictly increasing sequence of subspaces satisfying 
$$0=V_0 \subsetneq V_1 \subsetneq V_2\subsetneq \cdots \subsetneq V_i\subsetneq V_{i+1} \subsetneq \cdots \subsetneq V_r=V.$$
We say that this flag is of length $r$ (this counts the number of proper inclusions). To a flag $\mathcal F =(V_0, V_1, \cdots, V_r)$ we associate a sequence of positive integers $(n_1, n_2, \ldots, n_r)$ such that $\dim(V_i)=n_1+n_2\cdots +n_i$. Note that this sequence is a partition of $n$. The flag $(0, V)$ is called the trivial flag. It has length $1$ and corresponds to partition $n=n$. A flag is said to be complete if it has length $n$ and corresponds to the partition $n=1+1+\cdots +1$. We count unordered partitions in the sense that the partitions $1+2$ and $2+1$ of $3$ are distinct. Now, we fix a basis of $V$, say, $\{v_1,v_2,\ldots, v_n\}$. Then to each partition of $n=n_1+n_2+\cdots +n_r$ we associate a standard flag as follows: We define the subspace $V_i=<v_1,\ldots, v_{n_1+\cdots+n_i}>$. Conversely, every flag can be thought of this way with respect to some basis. Let us denote by $\mathfrak F$, the set of all flags.
\begin{exercise}
\begin{enumerate}
\item Show that $GL(V)$ acts on the set of all flags $\mathfrak F$ given by $g(V_1, V_2, \cdots, V_r)= (gV_1, gV_2, \cdots, gV_r)$.
\item Show that each orbit contains a standard flag. Hence the orbits are in one-one correspondence with partitions. How many orbits are there?
\item Determine the stabilisers in each case. Note that it is enough to determine this for the standard flag. We get the stair-case subgroups which in block matrix form looks like this
$$P_{(n_1,\ldots, n_r)} = \begin{pmatrix} GL_{n_1} & * & *& \cdots &* \\ 0 &GL_{n_2} & *&\cdots & * \\ \vdots & 0&\ddots & & *\\ \vdots & \vdots&&\ddots  & *\\0 & 0  &\cdots& 0& GL_{n_r} \end{pmatrix}.$$
This subgroup is called standard parabolic subgroups of $GL_n$.
\end{enumerate}  
\end{exercise}
\noindent The stabiliser of a flag is called a parabolic subgroup. The parabolic subgroup obtained by stabiliser of a standard flag is called standard parabolic. All other parabolics are conjugate of standard parabolics. The parabolic corresponding to a complete flag is called a Borel subgroup. Note that the standard Borel subgroup is the set of all upper triangular matrices. Further, every parabolic contains a Borel subgroup. 

\begin{exercise}[Levi decomposition]
Show that the standard parabolic subgroups have a following decomposition:
\begin{eqnarray*}
P_{(n_1,\ldots, n_r)}&=& \begin{pmatrix} GL_{n_1} & * & *& \cdots &* \\ 0 &GL_{n_2} & *&\cdots & * \\ \vdots & 0&\ddots & & *\\ \vdots & \vdots&&\ddots  & *\\0 & 0  &\cdots& 0& GL_{n_r} \end{pmatrix} \\
&=& \begin{pmatrix} GL_{n_1} & 0 & 0& \cdots &0 \\ 0 &GL_{n_2} & 0&\cdots & 0 \\ \vdots & 0&\ddots & & 0\\ \vdots & \vdots&&\ddots  & 0\\0 & 0  &\cdots& 0& GL_{n_r} \end{pmatrix}  \ltimes          \begin{pmatrix} I_{n_1} & * & *& \cdots &* \\ 0 & I_{n_2} & *&\cdots & * \\ \vdots & 0&\ddots & & *\\ \vdots & \vdots&&\ddots  & *\\0 & 0  &\cdots& 0& I_{n_r} \end{pmatrix}.
\end{eqnarray*}
This is obtained by defining a group homomorphism $P_{(n_1,\ldots,n_r)} \rightarrow \prod_{i=1}^r GL_{n_i}$. 
\end{exercise}
\noindent This kind of decomposition is not unique. We have simply exhibited one such. For example, in the case of Borel this amounts to fixing a maximal torus.

\section{Embedding field extensions in $GL_n(k)$}
Let $K$ be a field extension of $k$ of degree $n$. That is to say, considered $K$ over $k$ as a vector space is of dimension $n$. Let $\alpha\in K$. We define left multiplication map $l_{\alpha} \colon K \rightarrow K$ defined by $x\mapsto \alpha x$.
\begin{exercise}
\begin{enumerate}
\item What is the relation between minimal polynomial of the element $\alpha$ and the minimal polynomial of linear map $l_{\alpha}$?
\item What is the relation between the traces of $\alpha$ and $l_{\alpha}$?
\item What is the relation between the norm of $\alpha$ and determinant $l_{\alpha}$?
\end{enumerate}
\end{exercise}
\begin{exercise}
Prove that the map $\mathbb C \rightarrow M_2(\mathbb R)$ defined by $a+ib \mapsto \begin{pmatrix} a & -b \\ b&a\end{pmatrix}$ is an injective ring homomorphism.
\end{exercise}

\begin{exercise}[Trace Bilinear form]
Define the bilinear form on the $k$-vector space $K$ as follows: $B\colon K\times K \rightarrow k$ by $B(x,y)=trace(l_xl_y)$.
\begin{enumerate}
\item Show that $B$ is a symmetric bilinear form.
\item Show that $B$ is non-degenerate if and only if $K$ is a separable extension over $k$. 
\end{enumerate}
\end{exercise}
\noindent Since all field extensions of characteristics $0$ are separable the trace bilinear form is always non-degenerate. It's the case for finite extensions of finite fields. 
\begin{exercise}
Take the example of field $\mathbb F_p(t)$ and consider the extension given by splitting the polynomial $X^p-t$. Show that the trace bilinear form is degenerate.
\end{exercise} 

\section{Metric Topology}

Consider the set $M_n(\mathbb R)$ with Euclidean topology (identified with $\mathbb R^{n^2})$. 
\begin{exercise}
Decide if $GL_n(\mathbb R)$ and $SL_n(\mathbb R)$ are compact and/or connected. 
\end{exercise}

\begin{exercise}[Iwasawa decomposition]
Use Gram-orthogonalisation to prove that every matrix $A$ in $SL_n(\mathbb R)$ can be written as $A=PS$ where $P$ is an upper triangular matrix with all diagonals $>0$ and $S\in O_n(\mathbb R)$. 
\end{exercise}

\begin{exercise}
Consider $M_n(\mathbb C)$ (as $\mathbb C^{n^2}$) with usual metric topology. 
\begin{enumerate}
\item[(a)] Show that $GL_n(\mathbb C)$ is a dense open subset of $M_n(\mathbb C)$.
\item[(b)] Show that the set of all diagonalisable matrices in $GL_n(\mathbb C)$ is a dense subset.
\end{enumerate}
\end{exercise}

\section{Conjugacy classes in $GL_n(k)$}

When $k$ is an algebraically closed field (for example $k=\mathbb C$ or $\bar{\mathbb F_q}$) the conjugacy classes are determined by Jordan canonical forms. Let us recall this from linear algebra. We define the following: The linear transformations $T$ and $S$ on a finite dimensional vector space $V$ are said to be similar if there exists an invertible linear transformation $P$ such that $PTP^{-1}=S$. We note that if $T$ and $S$ are in $GL(V)$ they are similar if and only if they are conjugate. The theory of Jordan canonical forms uniquely (almost) associates a representative to each conjugacy class. 
\begin{theorem}[Jordan canonical forms]
Recall?
\end{theorem}

\begin{exercise}
Let $J_r(\lambda)$ be the $r\times r$ matrix with all diagonals $\lambda$ and $1$ at supper diagonals (above diagonal). 
\begin{enumerate}
\item Find minimal and characteristic polynomial of $J_r(\lambda)$?
\item Find minimal and characteristic polynomial of $J_{\bar r}(\lambda):=J_{r_1}(\lambda)\oplus J_{r_2}(\lambda)\oplus \cdots\oplus J_{r_k}(\lambda)$ where notation $\bar r=(r_1,\ldots,r_k)$ written in decreasing order is a partition.
\item Find minimal and characteristic polynomial of $J:=J_{n_1}(\lambda_1)\oplus J_{n_2}(\lambda_2)\oplus \cdots\oplus J_{n_m}(\lambda_m)$.
\item Find minimal and characteristic polynomial of $J:=J_{\bar n_1}(\lambda_1)\oplus J_{\bar n_2}(\lambda_2)\oplus \cdots\oplus J_{\bar n_m}(\lambda_m)$.
\end{enumerate}
\end{exercise}

The above definitions are extended to matrices. Two matrices $A$ and $B$ are said to be similar if there exists $P\in GL_n(k)$ such that $PAP^{-1}=B$. We note that the similarity relation on $GL_n(k)$ is same as conjugacy relation. However, the next question is to determine the conjugacy classes over any arbitrary field, specially when it is not algebraically closed. Since the eigen values do not exist we can not use the Jordan canonical form theory. Thus, one makes use of Rational canonical forms. 
\begin{exercise}
Recall rational canonical forms?
\end{exercise}

\begin{exercise}
Let $V$ be a vector space over field $k$ of dimension $\geq 2$. 
 \begin{enumerate}
\item If $k=\mathbb C$ every linear transformation has an eigen-vector. Show by an example that this may not be true over $\mathbb R$.
\item Not every linear transformation is diagonalisable. 
\item Over $\mathbb C$, every linear transformation is triangulable. Show by an example that this may not be true over $\mathbb R$. 
\item A linear transformation $T$ is called nilpotent if $T^r=0$ for some $r$. Every nilpotent linear transformation is triangulable. 
\item A linear transformation $T$ is called unipotent if $T-I$ is nilpotent. Every unipotent linear transformation is triangulable. 
\end{enumerate}
\end{exercise}

\begin{exercise}
Let $k$ be an algebraically closed field and $G\subset M_n(k)$ consisting of commuting set of matrices.
\begin{enumerate}
\item If all elements of $G$ are diagonalisable then prove that $G$ is simultaneously diagonalisable.
\item Give an example that this need not be true if $G$ is not a commuting set.
\item In general, show that $G$ is triangulable in such a way that all diagonalisable elements of $G$ become diagonal at the same time.
\end{enumerate}
\end{exercise}

\section{Jordan-Chevalley decomposition}
Let $V$ be a finite dimensional vector space over a field $k$. 
\begin{definition}
An element $X\in \End(V)$ is said to be {\bf semisimple} if it is diagonalisable over $\bar k$. An element $X\in \End(V)$ is said to be nilpotent if $X^r=0$ for some $r$. An element $X\in \End(V)$ is said to be unipotent if $X-1$ is nilpotent.
\end{definition}
\begin{exercise}
Let $X\in \End(V)$. Prove the following,
\begin{enumerate}
\item $X$ is nilpotent if and only if all eigen-values of $X$ are $0$. Prove that any such transformation can be always upper-tiangulaised.
\item $X$ is unipotent if and only if all eigen-values of $X$ are $1$. Prove that any such transformation can be always upper-tiangulaised and it is always invertible. 
\end{enumerate}
\end{exercise}
\begin{exercise}
Let $X\in \End(V)$. Suppose $W$ is an $X$-invariant subspace, i.e., $X(W)\subset W$. Prove the following:
\begin{enumerate}
\item If $X$ is semisimple then $X|_W$ is also semisimple.
\item If $X$ is nilpotent then $X|_W$ is also nilpotent.
\item If $X$ is unipotent then $X|_W$ is also unipotent.
\end{enumerate}
\end{exercise}
\begin{exercise}
Let $X\in \End(V)$. Suppose $W_1$ and $W_2$ are $X$-invariant subspaces such that $V=W_1\bigoplus W_2$. Prove the following:
\begin{enumerate}
\item If $X|_{W_1}$ and $X|_{W_2}$ both are semisimple then $X$ is also semisimple.
\item If  $X|_{W_1}$ and $X|_{W_2}$ both are nilpotent then $X$ is also nilpotent.
\item If $X|_{W_1}$ and $X|_{W_2}$ both are unipotent then $X$ is also unipotent.
\end{enumerate}
\end{exercise}
\noindent We can easily find examples and show that not every linear transformation is of one of the above form. However, it is very interesting to note that over algebraically closed field $k=\bar k$ every transformation is made of using only these transformations. We prove this below. This idea is useful in algebraic groups as when we requre to prove something for a group elementwise we first prove them for semisimple and unipotent elements and then try to guess if we can prove for a general element.
\begin{proposition}[Additive Jordan decomposition]
Let $k=\bar k$. Let $T\in \End(V)$. Then,
\begin{enumerate}
\item There exists $S,N\in \End(V)$ such that $T=S+N$ where $S$ is semisimple and $N$ is nilpotent satisfying the property that they commute, i.e., $SN=NS$.
\item The transformations $S$ and $N$ are unique with the above property.
\item There exists polynomials $p(X)$ and $q(X)$ without constants such that $S=p(T)$ and $N=q(T)$.
\end{enumerate}
\end{proposition}
\noindent We sketch the proof of this in the following exercises.
\begin{exercise}
Let $m_T(X)$ be the minimal polynomial of $T$ and its factorisation over $k=\bar k$ be
$$m_T(X) = \prod_{\alpha\in k} (X-\alpha)^{n_{\alpha}}.$$
Consider $V_{\alpha}=\{v\in V \mid (T-\alpha)^m(v)=0,\ for\ some\ m\geq 1\}$.
\begin{enumerate}
\item Show that $V=\bigoplus_{\alpha\in k} V_{\alpha}$ is a direct sum decomposition. 
\item Show that $V_{\alpha}=ker(T-\alpha)^{n_\alpha}$ and $T(V_{\alpha})\subset V_{\alpha}$.
\item Show that the minimal polynomial of $T_{\alpha}=T|_{V_{\alpha}}$ is $(X-\alpha)^{n_{\alpha}}$.
\item Let us construct $S$ now. Define a transformation $S$ on $V$ by defining it on each $V_{\alpha}$ to be the scalar transformation $\alpha.I_{V_{\alpha}}$, where $I_{V_{\alpha}}$ is identity on $V_{\alpha}$.
\begin{enumerate}
\item Prove that $S$ is semisimple.
\item $S$ and $T$ commute. Check this on each $V_{\alpha}$.
\item Show that $T-S$ is nilpotent on each $V_{\alpha}$.
\end{enumerate}
\item Define $N=T-S$ and prove that it is a nilpotent linear transformation.
\item Thus we have $T=S+N$. Now, show that $N$ and $S$ commute. This prove part 1.
\item Proof of part 2 is easy. Use the fact that only linear transformation which is semisimple and nilpotent both is the $0$ transformation.
\end{enumerate}
\end{exercise}
To complete the proof we still need to prove part 3. Which we do in the following,
\begin{exercise}
\begin{enumerate}
\item Prove that the polynomials $\{(X-\alpha)^{n_{\alpha}}\}$ are coprime. Then use Chinese Remainder Theorem to show that there exists a polynomial $p(X)$ such that $p(X) \equiv \alpha \imod{ (X-\alpha)^{n_{\alpha}}}$ for all $\alpha$ and $p(X) \equiv 0 \imod{X}$.  
\item Show that $S=p(T)$. 
\item Now define $q(X)=X-p(X)$ and show that $q(T)=N$.
\end{enumerate}
\end{exercise}
\begin{exercise}
Show that if $T$ is invertible then $S$ is also invertible. Further show that $S^{-1}$ is a polynomial in $T$. For this use that $p(X)$ and $m_T(X)$ are coprime. 
\end{exercise}
\noindent Using the above ideas conclude the following:
\begin{proposition}[Jordan decomposition]
Let $k=\bar k$. Let $T\in GL(V)$. Then,
\begin{enumerate}
\item There exists $S, U\in GL(V)$ such that $T=SU$ where $S$ is semisimple and $U$ is unipotent satisfying the property that they commute, i.e., $SU=US$.
\item The transformations $S$ and $U$ are unique with the above property.
\item There exists polynomials $p(X)$ and $q(X)$ without constants such that $S=p(T)$ and $U=q(T)$.
\end{enumerate}
\end{proposition}
The importance of this theorem in algebraic groups is this, Chevalley proved that Jordan decomposition  continues to hold for algebraic subgroups of $GL(V)$. Which is to say, that if we have an algebraic subgroup $G$ of $GL(V)$ and $T\in G$ then the components $T$ and $U$ are again in $G$. Further he proved that this is a functorial property, that is, semisimple and unipotent elements gets mapped to semisimple and unipotent elements respectively under algebraic group homomorphisms. For example the classical groups, such as, orthogonal groups and symplectic groups (over $k=\bar k$) are algebaric subgroups. Hence the Chevalley's theorem would imply that the Jordan decomposition of elements in these groups can be done with in those groups. 

\begin{exercise}
\begin{enumerate}
\item Let $char(k)=0$. Show that every finite order element in $GL_n(k)$ is semisimple.
\item Show that the matrices $\begin{pmatrix} \cos(\theta) & \sin(\theta) \\ -\sin(\theta) & \cos(\theta)\end{pmatrix}$ are semisimple in $SL_2(\mathbb C)$.
\end{enumerate}
\end{exercise}
\begin{exercise}
Consider the subgroup $H$ generated by the matrix $\begin{pmatrix} \pi &1 \\ \pi^{-1}\end{pmatrix}$ of $SL_2$. Does $H$ contain semisimple and unipotent parts of each of its elements?
\end{exercise}

\chapter{Root Datum}
In the theory of Lie algebras, we associate a combinatorial data to a semisimple Lie algebra, called {\bf root system}. The root systems are classified using their Dynkin diagram and in turn we classify (semi)simple Lie algebras. In the theory of Linear algebraic groups (over $k=\bar k$) we associate {\bf root datum} to a reductive group. In general root datum is more general than root system.

\section{Root Datum vs Root System}
\begin{definition}[Root Datum]
A root datum $\Psi=(X, R, Y, R^{\vee})$ is a quadrpule where
\begin{enumerate}
\item[(a)] $X$ and $Y$ are free Abelian groups of finite rank with a perfect pairing $<,>\colon X\times Y \rightarrow \mathbb Z$, and,
\item[(b)] $R$ and $R^{\vee}$ are finite subsets of $X$ and $Y$ respectively, along-with a bijection $R\rightarrow R^{\vee}$ denoted by $\alpha\mapsto \alpha^{\vee}$.
\end{enumerate}
This quadrpule satisfies the following properties:
\begin{enumerate}
\item $\forall \alpha, <\alpha, \alpha^{\vee}>=2$.
\item $s_{\alpha}(R)=R$ and $s_{\alpha}^{\vee}(R^{\vee})=R^{\vee}$ for all $\alpha \in R$, where $s_{\alpha}\colon X\rightarrow X$ is in $\Aut(X)$ given by $s_{\alpha}(x)=x-<x,\alpha^{\vee}>\alpha$ and $s_{\alpha}^{\vee}\in \Aut(Y)$ given by  $s_{\alpha}^{\vee}(y)=y-<\alpha,y>\alpha^{\vee}$.
\end{enumerate}
\end{definition}
Recall that a perfect pairing here is a $\mathbb Z$-bilinear map which induces isomorphism $X$ to $Y^*$ when we fix first variable, and $Y$ to $X^*$, when we fix the second variable.
The Weyl group of a root datum is defined as a subgroup $W(\Psi)=<s_{\alpha} \mid \alpha\in R>\subset \Aut(X)$. 
\begin{exercise}
$0\notin R$ and $W(\Psi)$ is finite.
\end{exercise}

\section{Root datum of rank $1$ semisimple algebraic groups}
Let $k$ be an algebraically closed field. Let $G$ be a semisimple algebraic group of rank $1$. There are two such groups: $SL_2(k)$ and $PGL_2(k)$. We try to understand their root datum and how they differ a bit from each other. 

\subsection{$SL_2(k)$} Let us first begin working with the group $SL_2(k)$. We fix the maximal torus $T=\left\{ \begin{pmatrix} t & \\ & t^{-1} \end{pmatrix} \mid t\in k^*\right\}$. The Lie algebra is $sl_2(k)=\{X\in M_2(k) \mid tr(X)=0\}$ and the $Ad$ action is given by $Ad \colon SL_2(k) \rightarrow GL(sl_2(k))$ where $Ad_g(X)=gXg^{-1}$. We restrict this action to $T$ and decompose the Lie algebra $sl_2(k)$ as simultaneous eigen-spaces. 
\begin{exercise}
Compute the character group $X(T)$ of the torus and show that it is isomorphic to $\mathbb Z$ generated by $\chi(\diag(t,t^{-1}))=t$.
\end{exercise}

\begin{exercise}
Show the decomposition $sl_2(k)=<h> \bigoplus <e_{12}> \bigoplus <e_{21}>$ with respect to the $Ad$ action of $T$. Here $h=\diag(1,-1)$. The following computation will be useful to conclude this:
\begin{enumerate}
\item $Ad(\diag(t,t^{-1}))(h)=h$,
\item $Ad(\diag(t,t^{-1}))(e_{12})=t^2e_{12} = 2\chi(t) e_{12}$,
\item $Ad(\diag(t,t^{-1}))(e_{21})=t^{-2}e_{21} = -2\chi(t) e_{21}$.
\end{enumerate}
\end{exercise}
\noindent Hence, $\Phi=\{2\chi, -2\chi\}\subset X(T)=<\chi>\cong \mathbb Z$. Verify that $<\Phi>\cong 2\mathbb Z \subset \mathbb Z$. Now, the question is to determine the co-roots. 
\begin{exercise}
Show that the co-character group $Y(T)$ of the torus is $<\psi>\cong \mathbb Z$ where $\psi\colon \mathbb G_m \rightarrow T$ given by $t\mapsto \diag(t,t^{-1})$. 
\end{exercise}
\begin{exercise}
For the root $\alpha=2\chi$ determine the co-root $\alpha^{\vee}\in Y(T)$ using the condition that it must satisfy $<\alpha, \alpha^{\vee}>=2$. 
\end{exercise}
\noindent Recall that the $<,> \colon X(T)\times Y(T) \rightarrow \mathbb Z$ is defined by $<m\chi, n\psi> = mn$, since $(m\chi)(n\psi)(t) = t^{mn}$. Hence we get $\alpha^{\vee}=\psi$ and $\Phi^{\vee}= <\psi>$. Thus, the root datum for the group $SL_2$ is 
$$(X(T), \Phi, Y(T), \Phi^{\vee}) = (<\chi>, \{2\chi\}, <\psi>, \{\psi\}) = (\mathbb Z, 2\mathbb Z, \mathbb Z, \mathbb Z).$$
Now we determine the root datum for $PGL_2$.

\subsection{$PGL_2(k)$}
The group 
$$PGL_2(k) = \frac{GL_2(k)}{ \mathcal Z(GL_2(k))} = \frac{GL_2(k)}{\{\diag(\lambda, \lambda)\mid \lambda\in k^*\}}.$$ 
We represent the maximal torus by $T=\{\diag(\lambda , 1)\mid \lambda\in k^*\}$ thought of as the image of diagonal torus in $GL_2(k)$ identified as follows: $\diag(\lambda_1, \lambda_2) \mapsto \diag(\lambda_1 \lambda_2^{-1}, 1)$.
\begin{exercise}
Show that the character group $X(T)=<\chi> \cong \mathbb Z$ where the character $\chi$ is given by $\chi(\diag(\lambda_1 \lambda_2^{-1}, 1)) = \frac{\lambda_1}{\lambda_2}$. 
\end{exercise}
The Lie algebra is $pgl_2= M_2(k)/\{\diag(\lambda, \lambda)\mid \lambda\in k\}$ and hence we can take $e_{11}, e_{12}, e_{21}$ as a basis of this Lie algebra.
Use the restriction of $Ad$ map to the maximal torus and do the simultaneous diagonalisation to prove the following.
\begin{exercise}
The root-space decomposition of the Lie algebra is
$$pgl_2 = <h> \bigoplus <e_{12}> \bigoplus <e_{21}>.$$
To prove this verify the following,
$Ad(\diag(\lambda_1, \lambda_2))(e_{11})=e_{11}$ and $Ad(\diag(\lambda_1, \lambda_2))(e_{12})=\frac{\lambda_1}{\lambda_2}e_{12}$.
\end{exercise}
Hence, $\Phi=\{\chi, -\chi\} \subset X(T)$ is the corresponding root system. We also note that $<\Phi> =<\chi> = X(T)\cong \mathbb Z$. Now the question is to find the co-roots.
\begin{exercise}
Show that the co-character group $Y(T)=<\psi>\cong \mathbb Z$ where $\psi\colon \mathbb G_m \rightarrow T$ given by $\psi(\lambda) \mapsto \diag(\lambda, 1)$. 
\end{exercise}
\noindent Now we need to compute the co-root $\chi^{\vee}$. This can be done by the equation $<\chi, a\psi>=2$ gives the value $a=2$. Hence $\chi^{\vee}=2\psi$. Hence the co-roots are $\Phi^{\vee}=\{2\psi, -2\psi\}$. Hence the root datum for the group $PGL_2(k)$ is 
$$(X, \Phi, Y, \Phi^{\vee}) = (<\chi>, \{\chi, -\chi\}, <\psi>, \{2\psi, -2\psi\}) \cong (\mathbb Z, \mathbb Z, \mathbb Z, 2\mathbb Z).$$
This example shows that the root datum of the two groups $SL_2(k)$ and $PGL_2(k)$ are different and they can be distinguished by knowing the root datum.

\section{Root datum of the reductive algebraic group $GL_n$}
In this section we have $k=\bar k$. Some of the exercises might be repetition from previous section but they have been included here for the sake of completeness.

\begin{exercise}
Show that $GL_n(k)$ is a Zariski-closed subset of $k^{n^2+1}$. Thus its an affine algebraic group of dimension $n^2$.
\end{exercise}
\begin{exercise}
\begin{enumerate}
\item Show that the set $D_n(k)$ of all diagonals in $GL_n(k)$ is a maximal torus.
\item Any two maximal tori in $GL_n(k)$ are conjugate.
\end{enumerate}
\end{exercise}
Hint: Since the centraliser of $D_n(k)$ is itself, the first one follows.
\begin{exercise}
The set of all semisimple elements 
$$G_s=\bigcup_{g\in GL_n}g D_n(k) g^{-1} = \bigcup_{S\ maximal\ tori} S. $$
\end{exercise}
\begin{exercise}
The set $N_n(k)$ of all upper triangular matrices with $1$ on diagonal is a maximal connected unipotent subgroup of $GL_n(k)$. In fact, any unipotent subgroup of $GL_n(K)$ can be conjugated to a subgroup of $N_n(k)$. Hence, the set of all unipotents $$G_u=\bigcup_{g\in GL_n} gN_n(k)g^{-1} .$$
\end{exercise}
\begin{exercise}
$$N_{GL_n(k)}(N_n(k)) = T_n(k)$$
where $T_n(k)$ is the set of all upper triangular matrices.
\end{exercise}
\begin{exercise}
Show that the group $T_n(k)\cong D_n(k) \ltimes N_n(k)$.
\end{exercise}
\begin{exercise}[Weyl group]
Consider $T=D_n(k)$, a maximal torus. Then,
\begin{enumerate}
\item $\mathcal Z_{GL_n(k)}(T)=T$.
\item $N_{GL_n(k)}(T)$ is the set of all monomial matrices.
\item The Weyl group 
$$W(GL_n, T)=\frac{N_{GL_n}(T)}{\mathcal Z_{GL_n}(T)} \cong S_n.$$
\item The Weyl group is independent of choice of a maximal torus.
\end{enumerate}
\end{exercise}
\noindent Recall that a matrix is said to be monomial if its each row and each column has exactly one non-zero entry. It can be also thought of as permutations of a diagonal matrix.

We recall the notation of unordered partition. A partition of $n$ is a sequence $(n_1,\ldots, n_r)$ such that $n=n_1+n_2+\cdots +n_r$. For us the partition $(1,2)$ and $(2,1)$ of $3$ are different.

\begin{exercise}
For a partition $n=n_1+n_2+\cdots + n_r$ consider the set 
$$T_{(n_1,\ldots, n_r)}=\{\diag(\lambda_1,\ldots, \lambda_1, \ldots, \lambda_r,\ldots, \lambda_r)\mid \lambda_1,\ldots, \lambda_r\in k^*\}$$ of diagonal matrices where $\lambda_i$ repeats $n_i$ many times. 
\begin{enumerate}
\item Show that $T_{(n_1,\ldots, n_r)}$ is a torus.
\item Any torus is conjugate to a subtorus of this form. 
\end{enumerate}
\end{exercise}

\begin{exercise}
Show that
$$\mathcal Z_{GL_n}(T_{(n_1,\ldots, n_r)}) \cong GL_{n_1}\times \cdots \times GL_{n_r}.$$
\end{exercise}
\begin{exercise}
\begin{enumerate}
\item Compute the normaliser $N_{GL_n(k)}(T_{(n_1,\ldots, n_r)})$.
\item Compute the quotient group 
$$W(GL_n, T_{(n_1,\ldots, n_r)})=\frac{N_{GL_n}(T_{(n_1,\ldots, n_r)})}{\mathcal Z_{GL_n}(T_{(n_1,\ldots, n_r)})}.$$
\end{enumerate}
\end{exercise}

For $i\neq j$ define the $ij^{th}$-elementary matrices $x_{ij}(t)=I + te_{ij}$.
\begin{exercise}[Root subgroups]
Show that the map $\psi_{ij} \colon \mathbb G_a \rightarrow GL_n(k)$ given by $t\mapsto x_{ij}(t)$ is a morphism of algebraic groups. 
\end{exercise}
\begin{exercise}
The group $SL_n(k)$ is generated by the set of all elementary matrices $\{x_{ij}(t) \mid i\neq j, t\in k\}$. In fact, the Gaussian elimination algorithm using row-column operations provides a proof of this.
\end{exercise}
\begin{exercise}
The commutator subgroup $[GL_n(k), GL_n(k)]=SL_n(k)$.
\end{exercise}
\begin{exercise}[Lie Algebra]
The set of all matrices $gl_n(k)$ and the set of all trace $0$ matrices $sl_n(k)$ are the Lie algebra of the group $GL_n(k)$ and $SL_n(k)$ respectively.
\end{exercise}
\begin{exercise}
Fix $i\neq j$. Show that the map $\phi_{ij}\colon SL_2(k)\rightarrow GL_n(k)$ defined by extending 
$$\begin{pmatrix} 1& t \\ &1\end{pmatrix} \mapsto x_{ij}(t),\hskip10mm  \begin{pmatrix} 1&  \\ t &1\end{pmatrix} \mapsto x_{ji}(t)$$
is a morphism of algebraic groups.
\end{exercise}
\begin{exercise}
With the notation as above,
\begin{enumerate}
\item Compute $n_{ij}(t)=x_{ij}(t)x_{ji}(-t^{-1})x_{ij}(t)$ and show that $n_{ij}(t)\in N_{GL_n}(D_n)$.
\item Compute $h_{ij}(t)=n_{ij}(t)n_{ij}(-1)$ and show that it belongs to $D_n(k)$.
\item The elements $n_{ij}:=n_{ij}(1)$ gives all possible row permutations and hence generates the Weyl group $W=\frac{N_{GL_n}(D_n)}{D_n}$.
\end{enumerate}
\end{exercise}
\begin{exercise}[Bruhat decomposition]
Let $G=GL_n$ and $B=T_n(k)$ be a Borel subgroup. Show that the double coset decomposition of $G$ by $B$ is 
$$G=\bigcup_{w\in S_n} BwB.$$
\end{exercise}
\noindent Notice that, in this case the double cosets have a group structure induced from the corresponding Weyl group. In general this is not true for an arbitrary finite group.

A character of an algebraic group $G$ is a morphism of algebraic group $\chi\colon G \rightarrow \mathbb G_m$. 
\begin{exercise}
\begin{enumerate}
\item Show that the determinant map is a character on $GL_n$.
\item Show that the character group $X(GL_n)\cong \mathbb Z$.
\end{enumerate}
\end{exercise} 

\begin{exercise}
Show that the set of all flags $\mathfrak F$ is a projective variety and hence complete. Thus the quotient $GL_n/P_{(n_1,\ldots,n_r)}$ is a complete variety. Hence the parabolic subgroups $P_{(n_1,\ldots, n_r)}$ are parabolic in algebraic group sense.
\end{exercise}

\begin{exercise}[Levi decomposition]
The parabolic subgroups $P_{(n_1,\ldots, n_r)}$ have decomposition $L_P\ltimes U_P$ where $L_P$ is called a Levi component which is a reductive group and $U_P$ is the unipotent radical of $P_{(n_1,\ldots,n_r)}$.
\end{exercise}
\begin{exercise}[Perfect Pairing]
Consider the maximal torus $T=D_n(k)$. 
\begin{enumerate}
\item The character group $X(T)=<\chi_1,\ldots, \chi_n> \cong \mathbb Z^n$ where $\chi_i\colon T \rightarrow \mathbb G_m$ is given by $\chi(\sum a_ie_{ii})=a_i$.
\item The co-character group $Y(T)=<\psi_1, \ldots, \psi_n> \cong \mathbb Z^n$ where $\psi_j\colon \mathbb G_m \rightarrow T$ given by $\psi_j(t)=te_{jj}$.
\item The pairing of $X(T)$ and $Y(T)$ is a map induced from the map 
$$< , >\colon X(T) \times Y(T) \rightarrow \Aut(\mathbb G_m) \text{\ given by\ } (\chi,\psi)\mapsto \chi\psi .$$ Thus, computing with the above notation this gives the map $<,>\colon \mathbb Z^n \times \mathbb Z^n \rightarrow \mathbb Z$ by $<(a_1,\ldots,a_n), (b_1,\ldots b_n)> = \sum a_ib_i$.
\end{enumerate}
\end{exercise}

\begin{exercise}\label{weylgroup}
The map $w\colon W(GL_n, T) \rightarrow \Aut(X(T))$ given by $n\mapsto w_n$ defined by 
$$(w_n(\chi))(t)=\chi(n^{-1}tn)$$
is an injective map.
\end{exercise}

\begin{exercise}
We consider $Ad$ action of $GL_n(k)$ on $M_n(k)$ as follows: $Ad_g(X)=gXg^{-1}$. The maps $Ad_g$ are vector space isomorphisms and hence we have $Ad\colon GL_n(k) \rightarrow GL(M_n(k))$. In fact, they are algebra isomorphism of $M_n(k)$.
\end{exercise}
\noindent We remark that the Lie algebra of the algebraic group $GL_n(k)$ is the set of all matrices $M_n(k)$ with Lie bracket defined by $[A,B]=AB-BA$.
\begin{exercise}
Let us consider the restriction of $Ad$ action to the maximal torus $T=D_n(k)$. Thus we have $Ad \colon T \rightarrow GL(M_n(k))$. Since $T$ consists of commuting set of semisimple elements, its image is simultaneously diagonalisable. Obtain the following simultaneous eigenspace decomposition  with respect to $T$,
$$M_n(k) = d_n(k) \bigoplus \bigoplus_{i\neq j} <e_{ij}>.$$
To get this prove the following:
\begin{enumerate}
\item For the trivial character $\chi=0$ (in additive notation which maps everything to the constant $1$) we get $V_{\chi}=d_n(k)$. This is because, $\{X\in M_n(k) \mid Ad_t(X)=X \forall t\in T\} = \mathcal Z_{M_n(k)}(T) = d_n(k)$.
\item For $i\neq j$, $V_{\chi_i-\chi_j}= <e_{ij}>$ where $\chi_i-\chi_j \in X(T)$ given by 
$$(\chi_i-\chi_j)(\diag(\lambda_1,\ldots, \lambda_n))=\lambda_i\lambda_j^{-1}.$$ 
For this we need to verify 
$$Ad_t(e_{ij}) = te_{ij}t^{-1} = (\chi_i-\chi_j)(t)e_{ij}.$$
\end{enumerate}
\end{exercise}
\noindent Now, we denote the set of non-zero characters appearing in the above decomposition by $\Phi=\{\chi_i-\chi_j \mid i\neq j\}\subset X(T)$.
\begin{exercise}
The set $\Phi$ has rank $n-1$, i.e., it spans $n-1$ dimensional Abeilan subgroup in $\mathbb Z^n$. Further, this can be obtained as kernel of the map $X(T) \rightarrow \mathbb Z$ given by $\sum a_i\chi_i \mapsto \sum a_i$.
\end{exercise}

\begin{exercise}
Take $\chi\in \Phi$, say $\chi=\chi_i-\chi_j$. 
\begin{enumerate}
\item Then show that $S= ker(\chi) = \{\diag(\lambda_1,\ldots ,\lambda_n) \in T \mid \lambda_i=\lambda_j\}$ is a torus of rank $n-1$.
\item Compute $G_{\chi}=\mathcal Z_{GL_n}(S)$ which for $i=1, j=2$ looks like 
$$= \left\{ \begin{pmatrix} GL_2 & & & & \\ &* &&& \\ &&\ddots&&  \\&&&&* \end{pmatrix} \right\}$$
and show that it is isomorphic to $GL_2 \times \mathbb G_m^{n-2}$.
\item Note that $S\subset T \subset G_{\chi}$. In fact, $S \subset \mathcal Z(G_{\chi})$.
\end{enumerate}
\end{exercise}

\begin{exercise}
Observe that $T\subset G_{\chi}$ is a maximal torus. Now define $W_{\chi}= \frac{N_{G_{\chi}}(T)}{\mathcal Z_{G_{\chi}}(T)}$, the Weyl group of $G_{\chi}$ with respect to $T$. Being isomorphic to the Weyl group of $G_{\chi}/S$ with respect to $T/S$, it is generated by $W_{\chi}=<n_{\chi}>$ where $n_{\chi}=x_{ij}(t)x_{ji}(-t^{-1})x_{ij}(t)$.
\end{exercise}

\begin{exercise}
For each $\chi\in \Phi$, we have $n_{\chi}$ obtained as above. Clearly, $n_{\chi} \in W(GL_n, T)$. Compute $w_{n_{\chi}}$ (as in Exercise~\ref{weylgroup}) and show that for $\chi=\chi_i-\chi_j$, it maps $\chi_i\leftrightarrow \chi_j$ and fixes all others. Thus, in $\Aut(X)$ it is a permutation of basis elements.
\end{exercise}

\begin{exercise}
Let us fix a Borel $B=T_n(k)$ containing our fixed diagonal maximal torus $T$. Then for $\chi=\chi_i-\chi_j$,
\begin{enumerate}
\item $G_{\chi} \bigcap N_n(k) = X_{ij}$, the root subgroup.
\item $G_{\chi} \bigcap N_n(k)^{-} = X_{ji}$. 
\end{enumerate}
\end{exercise}

Recall that the co-root $\alpha^{\vee}$ for $\alpha$ is defined by $<\alpha, \alpha^{\vee}>=2$ and $w_{n_{\alpha}}(x)=x-<x, \alpha^{\vee}>\alpha$ where $w_{n_{\alpha}}\in \Aut(X)$ is already determined.
\begin{exercise}[Co-roots]
Thus, for  $\chi=\chi_i-\chi_j$ we get $\chi^{\vee}=\psi_i-\psi_j$.
\end{exercise}
Hint: Let $\chi^{\vee}=\sum_{r} m_r\psi_r$. Now, we know $w_{n_{\chi}}(\chi_r)=\chi_{r}$ for $r\neq i, j$ which gives, $\chi_r= w_{n_{\chi}}(\chi_r)=\chi_r-<\chi_r, \sum m_s\psi_s>(\chi_i-\chi_j) = \chi_r - m_r(\chi_i-\chi_j)$ gives $m_r=0$. Now $\chi_j= w_{n_{\chi}}(\chi_i)=\chi_i-<\chi_i, \sum m_s\psi_s>(\chi_i-\chi_j) = \chi_i - m_i(\chi_i-\chi_j)$, gives $m_i=1$. Similarly, we get $m_j=-1$.

The quadruple $(X(T), \Phi, Y(T), \Phi^{\vee})$ obtained here is the root datum of $GL_n$ and determines it completely.

\chapter{(Real compact) Orthogonal group and Symmetry}

We begin with the set $\mathbb R^n$. To denote the elements of $\mathbb R^n$ we use coordinates. This coordinate system provides a unique representation of each element with respect to the coordinate system $\{e_1,\ldots,e_n\}$. This coordinate system is called the standard basis. The set $\mathbb R^n$ has many structures. 

The {\bf Euclidean space} $(\mathbb R^n, d)$ is a metric space with the distance function 
$$d(x,y):=\sqrt{(x_1-y_1)^2+\cdots +(x_n-y_n)^2}$$
where $x=(x_1,\ldots,x_n)$ and $y=(y_1,\ldots,y_n)$.
This induces a topology on $\mathbb R^n$. One can study functions with various properties such as continuity, differentiability etc. Usually we take $n\geq 2$. An {\bf isometry} of $\mathbb R^n$ is a map $f\colon \mathbb R^n\rightarrow \mathbb R^n$ such that $d(f(x),f(y))=d(x,y)$ for all $x,y\in\mathbb R^n$. Here we want to study isometries. We prove that isometries are bijection and form a group. They are also called affine linear maps.

We recall that $\mathbb R^n$ is an $n$-dimensional vector space over the field $\mathbb R$. Thus on $\mathbb R^n$ we already have ``nice'' linear maps. The coordinate system provides a standard basis. These linear maps can be represented as matrices with respect to the fixed standard basis. First, we would like to determine which linear maps are isometry. For this purpose we introduce norm on the vector space. The {\bf norm} on the vector space $\mathbb R^n$ is given by $||x||:=\sqrt{x_1^2+\cdots+x_n^2}$. We also have corresponding symmetric bilinear
form given by $B(x,y):=\sum_{i=1}^n x_iy_i$. The norm and symmetric bilinear form are related and can be obtained from each other.

\section{$O_2(\mathbb R)$ and $\mathbb C$}
The orthogonal group $O_2(\mathbb R)=\{A\in M_2(\mathbb R) \mid \tr AA=I\}$. 
\begin{exercise}
For $a\in \mathbb R^2$ define the translation map $\tau_a\colon \mathbb R^2\rightarrow \mathbb R^2$ by $\tau_a(x)=x+a$. Show that $\tau_a$ is an isometry but not a linear map. Further show
that $\tau_a\tau_b=\tau_{a+b}$ and $\tau_{a}^{-1}=\tau_{-a}$.
\end{exercise}

\begin{exercise}
Consider the rotation map $\rho_{\theta}\colon \mathbb R^2\rightarrow \mathbb R^2$ (around origin) given
by $$\rho_{\theta}(x_1,x_2)=(\cos(\theta)x_1-\sin(\theta)x_2, \sin(\theta)x_1+\cos(\theta)x_2).$$
\begin{enumerate}
\item Show that $\rho_{\theta}$ is an isometry. Also $\rho_{\theta}=\rho_{\theta+2\pi}$.
\item Show that $\rho_{\theta}$ is a linear, one-one and onto map. Write down
its matrix.
\item $\rho_{\theta_1}\rho_{\theta_2}=\rho_{\theta_3}$ where $\theta_3=\theta_1+\theta_2$ modulo $2\pi$. 
\end{enumerate}
\end{exercise}

\begin{exercise}
Given a non-zero vector $v\in\mathbb R^2$, consider the reflection map $r_v\colon \mathbb R^2\rightarrow \mathbb R^2$ defined by
$r_v(x)=x-2\frac{B(x,v)}{||v||^2}v$. This is a reflection in the line passing through origin perpendicular to the vector $v$.
\begin{enumerate}
\item Show that $r_v$ is an isometry and $r_v=r_{\alpha v}$ for any $\alpha\neq 0$.
\item Show that $r_v$ is a linear, one-one and onto map.
\item Show that $r_v$ is characterised (among linear maps) by $r_v(v)=-v$ and $r_v(x)=0$ for $x\perp v$ (i.e., $B(x,v)=0$).
\item Show that $r_v^2=1$. What do you get if you compose two reflections.
\end{enumerate}
\end{exercise}
\begin{exercise}
\begin{enumerate}
\item Use the definition and show that 
$$SO_2(\mathbb R)=\left \{\begin{pmatrix} \cos(\theta) & -\sin(\theta) \\ \sin(\theta) & \cos(\theta)\end{pmatrix} \mid 0\leq \theta < 2\pi \right\}.$$
These elements are called rotations. 
\item Use the map $\det \colon O_2(\mathbb R) \rightarrow \{\pm 1\}$ to show that $|\frac{O_2(\mathbb R)}{SO_2(\mathbb R)}|=2$.
\item By fixing an element $s\in O_2(\mathbb R)\backslash SO_2(\mathbb R)$ prove that $O_2(\mathbb R) = SO_2(\mathbb R) \bigcup s.SO_2(\mathbb R)$. Prove that we can take $s=\begin{pmatrix} 1 &  \\ & -1 \end{pmatrix}$.
\item Hence, 
$$O_2(\mathbb R)=\left \{\begin{pmatrix} \cos(\theta) & -\sin(\theta) \\ \sin(\theta) & \cos(\theta)\end{pmatrix}, \begin{pmatrix} \cos(\phi) & -\sin(\phi) \\ -\sin(\phi) & -\cos(\phi)\end{pmatrix} \mid 0\leq \theta, \phi < 2\pi \right\}.$$
\item Prove that every element of $O_2(\mathbb R)\backslash SO_2(\mathbb R)$ has a fixed line. These are reflections.
\item Show that $SO_2(\mathbb R)$ is an Abelian group but $O_2(\mathbb R)$ is not.
\item Compute $s\rho_{\theta}s^{-1}$ for $\rho_{\theta}\in SO_2(\mathbb R)$?
\item Prove that every element of $SO_2(\mathbb R)$ is a product of two  reflections. This is simply writing $\rho_{\theta}=s.s\rho_{\theta}$. 
\end{enumerate}
\end{exercise}

We explore its relation with complex numbers $\mathbb C$. It is a coincidence that the space $\mathbb R^2$ has a multiplication which makes it a field. The important fact is that the norm $N(x)=||x||^2$ is multiplicative with respect to the multiplication defined by the complex.
\begin{exercise}
Show that with complex multiplication $N(xy)=N(x)N(y)$. This will not be satisfied if we took co-ordinate wise multiplication on $\mathbb R^2$. 
\end{exercise}
\noindent Hurwitz proved that the only space which has multiplicative property for norm is $n=1,2, 4$ and $8$. Later in the next chapter we introduce Hamilton's quaternion on $\mathbb R^4$ and provide the next example. The multiplication on $\mathbb R^8$ which makes norm multiplicative makes it Octonion.
\begin{exercise}
Consider $\mathbb C$ as one dimensional vector over itself and $2$ dimensional vector space over $\mathbb R$ with basis $\{1,i\}$.
\begin{enumerate}
\item Let $0\neq \alpha \in \mathbb C$. Write down the real matrix of left multiplication by $\alpha$.
\item Show that $SO_2(\mathbb R)$ can be identified with the group $SO_2(\mathbb R)$.
\item Consider the conjugation map. What is its real matrix. 
\end{enumerate}
\end{exercise}

\begin{exercise}
\begin{enumerate}
\item Prove that any finite subgroup of $SO_2(\mathbb R)\cong S^1$ is cyclic.
\item Prove that any finite subgroup of $O_2(\mathbb R)$ which is not contained in $SO_2(\mathbb R)$ is a dihedral group.
\item Fix a prime $p$. Take the set $\{z\in S^1 \mid z^{p^r}=1 \ for \ some \ r\}$. Show that its a group.
\end{enumerate}
\end{exercise}

\begin{exercise}
Prove that $O_2(\mathbb R)$ is not connected whereas $SO_2(\mathbb R)$ is connected.  
\end{exercise}
\begin{exercise}
Prove that $O_2(\mathbb R)$ and $SO_2(\mathbb R)$ both are compact groups. 
\end{exercise}

\section{$O_n(\mathbb R)$}

A linear transformation $S\colon \mathbb R^n \rightarrow \mathbb R^n$ is called {\bf orthogonal}
if $||S(x)||=||x||$ for all $x\in \mathbb R^n$.
\begin{exercise}
Show that the following are equivalent for a linear transformation $S\colon 
\mathbb R^n \rightarrow \mathbb R^n$:
\begin{enumerate}
\item $S$ is orthogonal.
\item $S$ is an isometry.
\item $B(S(x),S(y))=B(x,y)$ for all $x,y\in \mathbb R^n$.
\item $S$ maps an orthonormal basis to an orthonormal basis.
\item The matrix $A$ of $S$ with respect to the standard basis satisfies $\tr AA=I$.
\end{enumerate}
\end{exercise}

We define the {\bf orthogonal group} $O(n)$ as the set of all orthogonal linear transformations.
In the matrix form (with respect to the fixed standard basis) 
$$O(n):=\{A\in M_n(\mathbb R)\mid \tr AA=I\}.$$
\begin{exercise}
Let $A, B\in O(n)$. Show that:
\begin{enumerate}
\item $\det(A)=\pm 1$.
\item $AB\in O(n)$ and $A^{-1}\in O(n)$.
\end{enumerate}
\end{exercise}

We define the set $SO(n):=\{A\in O(n)\mid \det(A)=1\}$.

\begin{exercise}
Use the reflection defined in the next exercise to show that $SO(n)\subset O(n)$ is proper.
\end{exercise}
\section{Isometries}

\begin{exercise}[Examples of isometry]
\begin{enumerate}
\item For $a\in \mathbb R^n$ we define translation $\tau_a\colon \mathbb R^n\rightarrow
\mathbb R^n$ by $\tau_a(x)=x+a$. Show that it's an isometry but not a linear map.
\item For a non-zero vector $v\in \mathbb R^n$ define reflection 
$r_v\colon \mathbb R^n \rightarrow \mathbb R^n$ by 
$$r_v(x)=x-2\frac{B(x,v)}{||v||^2}v.$$
Show that $r_v(v)=-v$ and $r_v$ fixes a $(n-1)$ dimensional hyperplane (which one!) pointwise. Further it
is an isometry and a linear map.
\item For $\theta\in [0,2\pi)$ and fixed $1\leq l\leq n-1$ we define 
$\rho_{l,\theta}\colon \mathbb R^n\rightarrow \mathbb R^n$
by $\rho_{l,\theta}(e_l)= \cos(\theta)e_l-\sin(\theta)e_{l+1}$, 
$\rho_{l,\theta}(e_{l+1})= \sin(\theta)e_l+\cos(\theta)e_{l+1}$ and
$\rho_{l,\theta}(e_i)=e_i$ for $i\neq l,l+1$. Show that $\rho_{l,\theta}$
is an isometry and is a linear map.
\end{enumerate}
\end{exercise}
\noindent However the distance function and the norm $||.||$, both are related on $\mathbb R^n$.
\begin{exercise}
Show the following:
\begin{enumerate}
\item $d(x,y)=||x-y||$. Further $d$ satisfies the $3$ properties of being a distance function.
\item $B(x,x)=||x||^2$.
\item $B(x,y)=\frac{||x+y||^2-||x||^2-||y||^2}{2}$. Further $B$ is a symmetric bilinear map.
\end{enumerate}
\end{exercise}
\begin{exercise}
Show that if $f\colon \mathbb R^n\rightarrow \mathbb R^n$ is a linear map which is also an isometry, then $f\in O_n(\mathbb R)$. 
\end{exercise}

\noindent Now we prove that any isometry is composition of a translation and an orthogonal map. We remark that it is not clear from the definition of an isometry that it is a bijection.
\begin{exercise}
Let $T\colon \mathbb R^n\rightarrow \mathbb R^n$ be an isometry. Suppose $T$ fixes origin
and also every element of the standard basis, i.e., $T(0)=0$ and $T(e_i)=e_i\forall i$.
Then show that $T$ is identity.
\end{exercise}
\begin{proof}
Let $x=(x_1,\ldots, x_n)\in \mathbb R^n$ and $T(x)=y=(y_1,\ldots, y_n)$. We want to prove $y=x$. We are given that $T$ is an isometry. Hence we have following equations, $d(0,x)=d(T(0),T(x))=d(0,y)$ implies $||x||=||y||$. Similarly,  $d(x,e_i)=d(T(x), T(e_i))=d(y,e_i)$ implies $||x-e_i||=||y-e_i||$. This gives the equation 
$$x_1^2+\cdots+x_{i-1}^2+(x_i-1)^2+x_{i+1}^2+ \cdots +x_n^2 = y_1^2+ \cdots + y_{i-1}^2+ (y_i-1)^2+ y_{i+1}^2 + \cdots +y_n^2$$
which after expanding and combining with $||x||=||y||$ gives $x_i=y_i$. This proves the required result. 
\end{proof}
\begin{exercise}
Let $T\colon \mathbb R^n\rightarrow \mathbb R^n$ be an isometry. Suppose $T$ fixes origin.
Then show that $T$ is an orthogonal linear transformation.
\end{exercise}
\begin{proof}
Consider a matrix $S$ of which columns are $y_i=T(e_i)$. Then $S$ defines a linear map $S\colon \mathbb R^n \rightarrow \mathbb R^n$ with the property $S(e_i)=y_i= T(e_i)$ for all $i$. We claim that $S$ is an orthogonal transformation. For this verify that $\{y_1,\ldots, y_n\}$ is an orthonormal basis, thus $S$ maps an orthonormal basis to an orthonormal basis.

Now let us consider the map $S^{-1}T\colon \mathbb R^n \rightarrow \mathbb R^n$. This is an isometry. Further this map fixes origin and all of the $e_i$. Hence from previous exercise it must be the identity map. This implies that $T=S$.
\end{proof}
\begin{exercise}
Let $f\colon \mathbb R^n\rightarrow \mathbb R^n$ be an isometry. 
Show that $f=\tau_a T$ where $a=f(0)$ and $T\in O_n(\mathbb R)$. 
\end{exercise}
\begin{proof}
Consider the map $\tau_{-a}f$ which is an isometry. Clearly $(\tau_{-a}f)(0)=\tau_{-a}(f(0))=f(0)-a=0$. Thus from previous exercise  $\tau_{-a}f$ is an isometry, say $T\in O_n(\mathbb R)$. Hence, $f=\tau_aT$.
\end{proof}
\noindent Thus we have the following,
\begin{exercise}
Let $f\colon \mathbb R^n\rightarrow \mathbb R^n$ be an isometry. 
\begin{enumerate}
\item Show that there exists $a\in \mathbb R^n$ and $A\in O_n(\mathbb R)$ such that $f(x)=Ax+a$. 
\item Let $f, f'$ be isometry given by $f(x)=Ax+a$ and $f'(x)=A'x+a'$. Show that 
$$(ff')(x) = AA'x+(Aa'+a).$$ 
\item Show that an isometry is one-one and onto map.
\item Show that the set of isometries $Iso(\mathbb R^n)$ is closed under composition. 
\item Show that the inverse of an isometry is again an isometry. In particular, $f^{-1}(x) = A^{-1}x - A^{-1}a$.
\item Show that $Iso(\mathbb R^n)$ is a group.
\item Show that $(f\tau_b f^{-1})(x) = \tau_{Ab}(x)$ where $f(x)=Ax+a$. Use this to show that the set of translations is a normal subgroup of $Iso(\mathbb R^n)$.
\item Prove that $Iso(\mathbb R^n)\cong \mathbb R^n \rtimes O_n(\mathbb R)$.
\end{enumerate}
\end{exercise}


\section{Symmetries and Platonic Solids}

\begin{exercise}
This exercise leads to the {\bf Classification of Platonic solids}.
\begin{definition}
A {\bf convex polyhedron} in 3d is said to be {\bf regular} if:
\begin{itemize}
\item Its faces are convex regular $n$-gons (for some $n \in \mathbb N, n \geq 3$) and any two faces are congruent.
\item Only possible intersections of faces are at edges.
\item At each vertex the same number (say $m$) of faces meet.
\end{itemize}
\end{definition}
It is clear that $n$ and $m$ are at least $3$. Regular convex polyhedra are also called Platonic solids. Cube is an example of a Platonic solid with $n = 4$ and $m = 3$.
For a Platonic solid the pair $(n,m)$ is called its Schl{\"a}fli symbol.  

Let $P$ be a Platonic solid with Schl{\"a}fli symbol $(n,m)$. Let $v$ be the number of vertices of $P$, $e$ be the number of edges of $P$ and $f$ be the number of faces of $P$.
Solve the following:
\begin{enumerate}
\item[(5a)] Show that $nf=2e=mv$.
\item[(5b)] Show that $\frac{1}{m}+\frac{1}{n}>\frac{1}{2}$.
\item[(5c)] Find all solutions $(n,m)$.
\end{enumerate}
\end{exercise}

\begin{exercise}
Compute the symmetry group of $n$-gons and Platonic solids. 
\end{exercise}

\section{Finite subgroups of $O_2(\mathbb R)$ and $O_3(\mathbb R)$}
These are some well known theorems worth attempting to give a prove.

\begin{theorem}
Any finite subgroup of $O_2(\mathbb R)$ is isomorphic to either the cyclic group $\mathbb Z/m\mathbb Z$ or the dihedral group $D_m$. 
\end{theorem}

\begin{theorem}
Any finite subgroup of  $O_3(\mathbb R)$ is isomorphic to one of the following: the cyclic group $\mathbb Z/m\mathbb Z$, the dihedral group $D_m$, $A_4, S_4$ or $A_5$. 
\end{theorem}

A more general theory at this point is trying to understand finite Coxeter groups. The finite Coxeter groups are defined as a subgroup generated by a bunch of finitely many reflections in an Orthogonal group. This is very welll understood and leads to classification via root systems, similar to simple Lie algebras. Interested reader can refer to the book by Humphreys on this subject.

\chapter{Hamilton's Quaternion}
The set $\mathbb H= \{ a= a_0+a_1 i + a_2 j +a_3 k \mid a_0, a_1, a_2, a_3 \in \mathbb R \}$, with addition and multiplication defined below, is called the set of real quaternions or Hamilton's quaternion. These operations make it an associative algebra (non-commutative ring as well as a vector space over $\mathbb R$). There is a much more general definition of quaternions over any field (not just over $\mathbb R$) which we describe in Chapter~\ref{quaternion}. The algebra operations on $\mathbb H$ are defined as follows: for $a=a_0+a_1 i + a_2 j +a_3 k$ and $b = b_0+ b_1 i + b_2 j +b_3 k$, we have,
\begin{eqnarray*}
a + b &=& (a_0+b_0) + (a_1+b_1) i + (a_2+b_2) j + (a_3+b_3) k \\
a.b &=& (a_0b_0 - a_1b_1 -a_2b_2 -a_3b_3) + (a_0b_1 + a_1b_0 + a_2b_3 -a_3b_2) i   \\ && + (a_0b_2 -a_1b_3 + a_2b_0 + a_3b_1) j + (a_0b_3 + a_1b_2 -a_2b_1 + a_3b_0)k.
\end{eqnarray*}
The multiplication can be simply obtained by using the following relations: $i^2=-1=j^2=k^2$ and $ij=-ji=k$. 

\begin{exercise}
Show that $\mathbb H$ is a division algebra. That is, it has all properties of field except the multiplication is not commutative. 
\end{exercise}

It has several properties similar to complex numbers $\mathbb C$. We define conjugation by 
$$\bar a = a_0-a_1 i - a_2 j- a_3 k.$$ 
\begin{exercise}
Compute and show that $a \bar a=a_0^2+a_1^2 + a_2^2 + a_3 ^2=\bar a a$ is a real number. 
\end{exercise}
This leads to the definition of norm, $N(a)=a\bar a$, and trace $tr(a)=a+\bar a = 2a_0$.

\begin{exercise}
\begin{enumerate}
\item Show that the map $N\colon \mathbb H^* \rightarrow \mathbb R^*$ defined by $a\mapsto N(a)$ is a multiplicative group homomorphism. 
\item Show that the map $Tr \colon \mathbb H \rightarrow \mathbb R$ defined by $a \mapsto Tr(a)$ is an $\mathbb R$-linear map.
\end{enumerate}
\end{exercise}

\begin{exercise}
\begin{enumerate}
\item How many solutions does the equation $X^2 + 1$ have over $\mathbb H$? Compare your answer with solving equations over a field.
\item For an $a\in \mathbb H$ show that it always satisfies the quadratic equation $X^2 - (tr(a)) X + N(a) =0$.
\end{enumerate}
\end{exercise}

\section{Realization of quaternions}
There are several ways to think about quaternions. In various situation they can be helpful in computation.

\subsection{As double complex numbers}
For $a\in \mathbb H$ we write 
$$a=a_0+a_1 i + a_2 j +a_3 k = (a_0+a_1 i) + (a_2+a_3 i)j = z_1 + z_2 j$$
where $z_1, z_2$ can be thought of as elements in $\mathbb C$. Thus we may think of $\mathbb H = \mathbb C + \mathbb C j$ and note that the new multiplication formula would be, for $a=z_1 + z_2 j$ and $b = w_1 + w_2 j$,
$$a.b = (z_1 + z_2 j).(w_1 + w_2 j) = (z_1w_1-\bar w_2 z_2) + (w_2z_1 + z_2\bar w_1)j.$$
We remark that this formula can be further generalised to define an octonion algebra. 

\begin{exercise}
Verified that the new formula is same as usual multiplication on $\mathbb H$. In this notation, $N(a)= N(z_1 + z_2 j) = N(z_1)+N(z_2)$ and $\bar a = \overline{z_1+z_2 j} = \bar z_1 - z_2 j$ where the norm and conjugation on complex numbers are defined in usual way. 
\end{exercise}

\subsection{As $2\times 2$ complex matrices}

For $a\in \mathbb H$ we write $a= z_1 + z_2 j$ where $z_1, z_2 \in \mathbb C$ as above. Now we identify  $a$ with the matrix $a=\begin{pmatrix} z_1 & -z_2 \\ \bar z_2 & \bar z_1 \end{pmatrix}$ in $M_2(\mathbb C)$. Thus we can think of $\mathbb H$ to be the set of $2\times 2$ complex matrices of this kind. The multiplication in this case is simply given by the matrix multiplication.
\begin{exercise}
Verify that the quaternion multiplication is same as matrix multiplication represented in this form, i.e, 
$$a.b = \begin{pmatrix} z_1 & -z_2 \\ \bar z_2 & \bar z_1 \end{pmatrix} \begin{pmatrix} w_1 & -w_2 \\ \bar w_2 & \bar w_1 \end{pmatrix} = \begin{pmatrix} z_1w_1 -z_2\bar w_2 & -z_1w_2-z_2\bar w_1 \\ \bar z_2w_1+ \bar z_1\bar w_2 & -\bar z_2 w_2 + \bar z_1\bar w_1 \end{pmatrix}.$$
\end{exercise}
\begin{exercise}
The norm $N(a)=\det\begin{pmatrix} z_1 & -z_2 \\ \bar z_2 & \bar z_1 \end{pmatrix} = z_1\bar z_1 + z_2\bar z_2$ and conjugation is $\bar a = \begin{pmatrix} \bar z_1 & z_2 \\ -\bar z_2 & z_1 \end{pmatrix} = \tr(\bar a)$. Also, trace is simply trace of the matrix.
\end{exercise}
Remember the following association in the matrix notation:
$$1=\begin{pmatrix} 1 & \\ & 1 \end{pmatrix}, i=\begin{pmatrix} i & \\ & -i \end{pmatrix}, j=\begin{pmatrix}  & -1 \\ 1 &  \end{pmatrix}, j = \begin{pmatrix}  & -i \\ -i & 0 \end{pmatrix}.$$

\begin{exercise}[Four Square Identity]
\begin{enumerate}
\item Show that $\overline{ab}=\bar b \bar a$, and hence, $N(ab) = ab. \overline{ab} = ab \bar b \bar a = a N(b) \bar a = N(a)N(b)$.
\item Since the norm map is multiplicative, ie., $N(ab)=N(a)N(b)$, write this in expanded form to get the Four Square Identity.
\begin{eqnarray*}
&&(a_0^2+a_1^2 + a_2^2 +a_3^2). (b_0^2+b_1^2 + b_2^2 +b_3^2)  \\
&=&   (a_0b_0 - a_1b_1 -a_2b_2 -a_3b_3)^2 + (a_0b_1 + a_1b_0 + a_2b_3 -a_3b_2)^2   \\ && + (a_0b_2 -a_1b_3 + a_2b_0 + a_3b_1)^2  + (a_0b_3 + a_1b_2 -a_2b_1 + a_3b_0)^2.
\end{eqnarray*}
\item Show that, over $\mathbb Z$, the product of a sum of $4$-squares is again a sum of $4$-squares.
\end{enumerate}
\end{exercise}

\section{Application to Number Theory}

One of the famous problem in number theory is Waring problem. It asks the following question. Given a natural number $r$, does there exists a natural number $f(r)$, such that the polynomial function $x_1^r+x_2^r +\cdots + x_{f(r)}^r$ takes value all positive integers if we substitute integers. If yes, find the smallest value $f(r)$. For $r=2$, the answer is $f(2)=4$ which is the four square theorem. There are several proofs of this result and here we present one. 

\begin{theorem}[Four Square theorem]
Every natural number is a sum of four integer squares. 
\end{theorem}
\begin{lemma}
Let $p$ be an odd prime, say $p=2n+1$. Then there exists $l,m \in \mathbb Z$ such that $p | (1+l^2+m^2)$. 
\end{lemma}
\begin{proof}
We know that $\mathbb Z/p\mathbb Z$ is a field. Consider the set $[n]=\{0,1,\ldots, n\}$. Let $x\neq y \in [n]$. Then $x^2\not\equiv y^2 \imod p$. For if $x^2\equiv y^2 \imod p$ then $(x^2-y^2)\equiv 0\imod p$ and hence either $x-y=0$ or $x+y=0$ in $\mathbb Z/p\mathbb Z$. Since $0< x+y < n+n < p$ we can't have $x+y=0$.

Hence $T=\{0^2, 1^2, \ldots, n^2\}$ are all distinct mod $p$. Similarly the set $-1 - T=\{-1-0^2, -1-1^2, -1-2^2, \ldots, -1-n^2\}$ has all elements distinct. Both sets $T$ and $-1-T$ have $n+1$ elements and hence they have a common element. That is there exists $l^2\in T$ and $-1-m^2 \ in -1-T$ such that $l^2=-1-m^2$ and we are done. 
\end{proof}

\begin{proof}[Proof of the theorem]
Let $n$ be a natural number. One can easily verify this theorem for $n=0, 1, 2, 3, 4$. Thus we may assume $n\geq 5$.

Step 1. We can reduce this theorem to prove only for primes. For if $n$ is not a prime, write $n=n_1 n_2$ and use the four square identity. Since $n_1$ and $n_2$ both are sum of four squares, so is the product $n$.

Step 2. If $2m$ is a sum of two squares then so is $m$. Let $2m=x^2+y^2$ then either both $x,y$ are even or both are odd. Then write $m=(\frac{x-y}{2})^2 + (\frac{x+y}{2})^2$.

Step 3. Let $p$ be an odd prime. Then from the lemma a multiple of $p$ is a sum of four squares. That is, there exists $l,m$ such that $p | (1+l^2+m^2)$ and hence $kp = 1 +l^2 +m^2$ for some $k$. Further we assume $0<k<p$ (else we may go mod $p$). Thus if we already have $k=1$ we are done. Otherwise we do the following.

Step 4. We apply decent argument to arrive at a contradiction. Suppose there exists $m , a, b ,c, d$ integers such that $mp=a^2 + b^2 + c^2 + d^2$ and $1< m < p$. We show that we can find $n$ such that $1\leq n < m$ and $np=a'^2 + b'^2 + c'^2 + d'^2$ for some $a', b', c',d'$. This would provide contradiction if we begin with minimal $m$ with the required property. 

Without loss of generality $m>1$. If $m$ is even, say $m=2n$, then $mp$ is even which is a sum of $4$ squares and hence step 2 implies that $np$ is a sum of four squares and we are done. Now we assume that $m$ is odd $>2$. Now pick $w, x, y, z$ such that $w\equiv a\imod m$, $x\equiv b\imod m$, $y\equiv c\imod m$ and $z\equiv d\imod m$ and $-\frac{m}{2} < w,x,y,z < \frac{m}{2}$. Thus, $w^2+x^2+y^2+z^2 < 4.\frac{m^2}{4} = m^2$ and $w^2+x^2+y^2+z^2 \equiv a^2+b^2+c^2+d^2 \equiv 0 \imod m$. Hence $w^2+x^2+y^2+z^2=nm$ for some $n$. Since $w^2+x^2+y^2+z^2< m^2$ we have $n<m$. Check that $n\neq 0$. 

Thus, $(a^2+b^2+c^2+d^2)(w^2+x^2+y^2+z^2)=mp.nm = nm^2p$ is a sum of $4$ squares. Expanding out the left side (and replacing $x,y,z$ by their negative) we get $ (aw + bx +cy +dz)^2 + (-ax + bw - cz +dy)^2  + (-ay +bz + cw - dx)^2  + (-az - by + cx + dw)^2 = m^2.np$. Claim: $m$ divides each term on left, $(aw + bx +cy +dz), (-ax + bw - cz +dy), (-ay +bz + cw - dx), (-az - by +cx + dw)$. For this, first we write each of this term mod $m$. We get, 
$(aw + bx +cy +dz) = a^2+b^2+c^2+d^2 \imod m =0\imod m, (-ax + bw - cz +dy) = 0 \imod m, (-ay +bz + cw - dx)=0 \imod m, (-az - by +cx + dw)= 0 \imod m$. Hence, by dividing on the both side in the equation by $m^2$ we get that $np$ is a sum of $4$-squares. This proves the required result.
\end{proof} 

\begin{exercise}
Can you write every natural number as a sum of $2$ or $3$ squares? What about sum of $2$ or $3$ cubes?
\end{exercise}

\section{Application to Geometry}

Let $\mathbb H =\{a=a_0+a_1i+a_2j+a_3k \mid a_r\in \mathbb R, i^2=-1=j^2, ij=-ji=k\}$ be the real quaternions defined above. Recall that the conjugation is defined by $\bar a = a_0-a_1i-a_2j-a_3k$. We also denote $tr(a)=a+\bar a$ and $N(a)=a\bar a$. Now, we can think of $\mathbb H$ as a $4$-dimensional vector space over $\mathbb R$. Further we have a non-degenerate symmetric bilinear form $B$ on $\mathbb H$ induced from $N$ given as follows: 
\begin{eqnarray*}
B(x,y)&=& \frac{N(x+y)-N(x)-N(y)}{2}\\
&=& \frac{1}{2}[ (x_0+y_0)^2 + (x_1+y_1)^2+ (x_2+y_2)^2 + (x_3+y_3)^2 -   (x_0^2 +x_1^2+x_2^2+x_3^2)\\&& - (y_0^2 +y_1^2+y_2^2+y_3^2)] \\
&=& x_0y_0+x_1y_1+ x_2y_2 + x_3y_3.
\end{eqnarray*}
The three dimensional subspace spanned by $i,j,k$ is called imaginary subspace. 
\begin{exercise}
Let $V:= Im(\mathbb H)=<i,j,k>$ be the vector subspace of $\mathbb H$.
\begin{enumerate}
\item[(a)] For $a=\cos(\theta)j+\sin(\theta)k$, compute the matrix of the map $\phi_a\colon V\rightarrow V$ defined by $\phi_a(v)=ava^{-1}$. Show that $\phi_a = \begin{pmatrix} -1 && \\ &\cos(2\theta) & \sin(2\theta) \\ &\sin(2\theta) & -\cos(2\theta) \end{pmatrix}$.
\item[(b)] Compute the matrix of the map $\psi\colon V\rightarrow V$ given by $\psi(v)=ivi$. 
Show that $\psi$ is a reflection. The matrix of $\psi$ is $\diag(-1, 1, 1)$.
\end{enumerate}
\end{exercise}

\begin{exercise}
Use the group homomorphism given by norm $N\colon \mathbb H^*\rightarrow \mathbb R^*$ defined by $a\mapsto N(a)=a\bar a$ to show that $\mathbb H^1:= ker(N)$ is a group. Hence, show that the sphere $\mathbb S^3$ is a group.
\end{exercise}

\noindent It is an amazing fact that the sphere $\mathbb S^n \subset \mathbb R^{n+1}$ is a group if and only if $n=1, 3$ or $7$. This phenomena is related to being able to put composition algebra structure on $\mathbb R^{m}$ which is possible if and only if $m=1, 2, 4$ or $8$. The composition algebras of dimension $1$ is the field $\mathbb R$, of dimension $2$ is either $\mathbb R\times \mathbb R$ or $\mathbb C$, of dimension $4$ is either $M_2(\mathbb R)$ or $\mathbb H$ and of dimension $8$, two possibilities, called octonions. 
\begin{exercise}
For a non-zero $a\in \mathbb H$ define a map $\phi_a\colon \mathbb H\rightarrow \mathbb H$
by $\phi_a(x)=axa^{-1}$. Prove the following:
\begin{enumerate}
\item Show that $\phi_a$ preserves norm, hence is in $O_4(\mathbb H, N)$.
\item Show that $\phi_a(1)=1$.
\item Show that if $x\perp 1$ then $\phi_a(x) \perp 1$. Thus $\phi_a$ restricts to $1^{\perp}=V$.
\end{enumerate}
\end{exercise}
\noindent Thus, we define a map $a\mapsto \phi_a$ from $\mathbb H^1$ to $SO_3(\mathbb R)$ and we get the following exact sequence:
$$1\rightarrow \{\pm 1\} \rightarrow \mathbb H^1 \rightarrow SO_3(\mathbb R) \rightarrow 1.$$
\begin{exercise}
Consider $V=Im(\mathbb H)$. Compute the conjugation map $\phi_a\colon V\rightarrow V$ for
$a=a_0+a_1i$ and $a=a_2j+a_3k$.
\end{exercise}
\begin{exercise}
Show that the map $\phi \colon \mathbb H^1 \rightarrow SO_3(\mathbb R)$ defined by $a\mapsto \phi_a$ is a surjective group homomorphism with kernel $\{\pm 1\}$. Thus $\mathbb H^1$ is a double cover of the group $SO_3(\mathbb R)$.  
\end{exercise}
\noindent In the above exercise, if needed, you may use connectedness.
\begin{exercise}
Show that an element of $SO(3)$ always fixes a non-zero vector, i.e., $1$ is an eigenvalue. Thus up to conjugacy it is of the following form 
$$\begin{pmatrix}1&&\\&\cos(\theta) &-\sin(\theta)\\ & \sin(\theta)&\cos(\theta) \end{pmatrix}.$$
\end{exercise}

\begin{exercise}
For $a\in \mathbb H^1$ define the map $l_a\colon \mathbb H \rightarrow \mathbb H$ by $x\mapsto ax$. Show that its an isometry. This defines a group homomorphism $l\colon \mathbb H^1 \rightarrow O_4(\mathbb R)$.
\end{exercise}

\section{The group $SU_2(\mathbb C)$} 
The group $SU_2(\mathbb C) = \{X\in M_2(\mathbb C) \mid \tr{\bar X} X=I, \det(X)=1\}$. 
\begin{exercise}
Compute and show that $X\in SU_2(\mathbb C)$ is of the form $\begin{pmatrix} w & -z \\ \bar z & \bar w\end{pmatrix}$ for some $w, z\in \mathbb C$ satisfying $w\bar w + z\bar z =1$. 
\end{exercise}

\begin{exercise}
\begin{enumerate}
\item Recall the realization of $\mathbb H$ as $2\times 2$ complex matrices and note that $\mathbb H^1$ is identified with $SU_2(\mathbb C)$.
\item Show that $\mathbb H^1\cong S^3$ is group isomorphic to $SU_2(\mathbb C)$.
\item Use this to get the double cover map $SU_2(\mathbb C)\rightarrow SO_3(\mathbb R)$.
\begin{center}
\begin{tikzcd}
\mathbb H^1 \arrow[rr] \arrow[rd]&  & \arrow[dl]SU_2(\mathbb C) \\ 
&SO_3(\mathbb R)&
\end{tikzcd}
\end{center}
\end{enumerate}
\end{exercise}

\begin{exercise}
Show that $SU_2(\mathbb C)$ is simply connected and $SO_3(\mathbb R)\cong \mathbb RP^3$.  
\end{exercise}

Here we describe the topology of $SO_3(\mathbb R)$. Let us first recall how we get projective spaces. We begin with $\mathbb RP^2$. This is suppose to be obtained by antipodal identification of $S^2$. We can further think of this as the upper hemisphere with boundary circle being identified, which after projecting down in the plane becomes solid circle (disc) in the plane with points on the boundary being identified. We can carry forward this idea to any higher dimension.

Now, any element $A\in SO_3(\mathbb R)$ has a fixed axis and it is rotation in the perpendicular plane. So, $A$ can be associated to $(v,\theta)\in S^2\times [-\pi, \pi]$ where the first coordinate $v$ represents a direction and the second coordinate represents the rotation amount. Notice that the rotation by $\pi$ and $-\pi$, in this representation, is same and hence in this representation we have to further identify $\pi$ with $-\pi$. Thus the object, $S^2\times [-\pi, \pi]$, is a solid sphere (disc) of radius $\pi$ together with the property that points on the surface are identified. This object is nothing but $\mathbb RP^3$.


\chapter{Symplectic groups}

Let $V$ be a vector space of dimension $2l$ over a field $k$. We also assume $char(k)\neq 2$. There is a unique non-degenerate skew-symmetric linear form $\beta$ up to equivalence on $V$. We fix a basis of $V$ indexed as follows: $\{e_1, \ldots, e_l, e_{-1}, \ldots, e_{-l}\}$ which gives the following skew-symmetric matrix $J$ corresponding to the form $\beta$:
$$J=\begin{pmatrix} &I_l \\-I_l & \end{pmatrix}.$$
The symplectic group is the set of isometries of the form $\beta$, which in matrix form is the set of symplectic matrices with respect to this form:
$$Sp_{2l}=\{X\in GL_{2l} \mid \tr XJX = J\}.$$
In the algebraic groups notation, this is a group of type $C_l$. We also define, the projective symplectic group, 
$$PSp_{2l}=\frac{Sp_{2l}}{\mathcal Z(Sp_{2l})}.$$

\begin{exercise}
When $l=1$, show that $Sp_2=SL_2$.
\end{exercise}
Thus, we always assume $l>1$. 
\begin{exercise}
For $g\in Sp_{2l}$, we have $\det(g)=1$. Hence, $Sp_{2l}$ is a subgroup of $SL_{2l}$.
\end{exercise}
Hint: Use Pfaffian to prove this. The Pfaffian $Pf(A)$ is a polynomial associated to a skew-symmetric matrix $A$ which is non-zero when the matrix is invertible of even size. It also has the propery that $Pf(\tr X A X) = \det(X) Pf(A)$. 

\begin{exercise}
Show that the map $GL_l \rightarrow Sp_{2l}$ given by $A\mapsto \begin{pmatrix} A & \\ & \tr A^{-1}\end{pmatrix}$ is an injective group homomorphism.
\end{exercise}

\begin{exercise}
Compute the center $\mathcal Z(Sp_{2l})=\{\pm I\}$.
\end{exercise}

\section{Symplectic Lie algebra}
Notice that we are going to use $\{1, \ldots, l, -1, \ldots, -l\}$ as indices instead of using $1$ upto $2l$. The Lie algebra of the symplectic group $Sp_{2l}$ is 
$$\mathfrak{sp}_{2l} = \{X\in \mathfrak{gl}_{2l} \mid \tra XJ = -JX\}$$
where $\mathfrak{gl}_{2l}$ is the set of $2l$ size matrices.
\begin{exercise}
For $X\in \mathfrak{sp}_{2l}$ show that $tr(X)=0$.
\end{exercise}
\begin{exercise}
Let $\begin{pmatrix} W&X \\Y & Z \end{pmatrix}$ be in $\mathfrak{sp}_{2l}$ written in block form of size $l$. Show that $Z=-\tra W$, $X=\tra X$ and $Y=\tra Y$. 
\end{exercise}
Hint: We can do block multiplication here.
\begin{exercise}\label{basis-sp}
Show that the following set of elements for $1\leq i, j \leq l$, 
$$\{(e_{i,j}-e_{-j,-i}) \text{\ for}\; i\neq j, \} \cup \{(e_{i,-j}+e_{j,-i}) \text{\ for}\; i<j,\} \cup  
\{(e_{-i,j}+e_{-j,i}) \text{\ for}\; i<j\}, \cup \{e_{i,-i}, e_{-i,i}\}
$$ 
form a basis of the Lie algebra together with the diagonals $\{e_{i,i} - e_{-i,-i}\}$.
\end{exercise}

\begin{exercise}
Compute the dimension of the Lie algebra $\mathfrak{sp}_{2l}$.
\end{exercise}

\begin{exercise}
Show that each of the matrices in the basis above (except the diagonal ones) is nilpotent.
\end{exercise}

\begin{exercise}[Chevalley generators]\label{ch-gen-sp}
The following elements belong to the symplectic group for any $t\in k$ and for any $1\leq i,j \leq l$:
\begin{eqnarray*}
x_{i,j}(t)&=& I +t(e_{i,j}-e_{-j,-i}) \text{\ for}\; i\neq j, \\ 
x_{i,-j}(t)&= &I +t(e_{i,-j}+e_{j,-i}) \text{\ for}\; i<j,  \\ 
x_{-i,j}(t)&=& I +t(e_{-i,j}+e_{-j,i}) \text{\ for}\; i<j, \\ 
x_{i,-i}(t) &=& I +t e_{i,-i}, \\ 
x_{-i,i}(t)&= & I+te_{-i,i}.
\end{eqnarray*}
In fact, these matrices generate the symplectic group. This is a non-trivial exercise but not very difficult. See https://arxiv.org/pdf/1504.03794.pdf for an algorithmic proof similar to row-column operations in the $SL_n$ case.
\end{exercise}

\begin{exercise}
Assume that the Chevallley generators in the above exercise~\ref{ch-gen-sp} generate the symplectic group. Use this to prove the following:
\begin{enumerate}
\item For $X\in Sp_{2l}$, we have $\det(X)=1$.
\item For $X\in Sp_{2l}$, we have $\tr X\in Sp_{2l}$.
\end{enumerate} 
\end{exercise}

\section{Root Datum}
\begin{exercise}
The set $T=\{\diag(\lambda_1, \ldots, \lambda_l, \lambda_1^{-1}, \ldots, \lambda_l^{-1})\mid \lambda_i\in k^*\}$ is a maximal torus in $Sp_{2l}$. Thus, rank of $Sp_{2l}$ is $l$.
\end{exercise}

\begin{exercise}
Fix characters $\chi_i \colon T \rightarrow \mathbb G_m$ by $\chi_i(\diag(\lambda_1, \ldots, \lambda_l, \lambda_1^{-1}, \ldots, \lambda_l^{-1}))= \lambda_i$. Show that the character group $X^*(T)\cong \mathbb Z^l$.
\end{exercise}

\begin{exercise}
Fix co-characters $\mu_i \colon \mathbb G_m \rightarrow T$ by $\chi_i(t) = \diag(1, \ldots, t\ldots ,1 \ldots, \ldots , 1)$ where $t$ is at $i^{th}$ place. Show that the co-character group $X_*(T)\cong \mathbb Z^l$.
\end{exercise}

\begin{exercise}
Consider the $Ad \colon Sp_{2l} \rightarrow GL(\mathfrak{sp}_{2l})$ given by $Ad(g)(X)=gXg^{-1}$. We restrict this map to the maximal torus $T$ and obtain the simultaneous decomposition of $\mathfrak{sp}_{2l}$,
$$\mathfrak{sp}_{2l}=\mathfrak t \bigoplus \left ( \bigoplus_{\chi\in X^*(T)} \mathfrak g_{\chi} \right ).$$
Show that the diagonals $\mathfrak t$ form a $T$-invariant subspace with respect to the $0$ character. Show that each of the basis vectors in the exercise~\ref{basis-sp}, other than the diagonals, form a $1$ dimensional $T$-invariant subspace $\mathfrak g_{\chi}$.
\end{exercise}
Hint: Use the action to compute the eigen-values which are the charcters in the exercise below. Since these turn out to be distinct they are, each one of them, all possible eigen-spaces.
\begin{exercise}
Compute the set of all non-trivial characters appearing in the above decomposition $P\subset X^*(T)$. Show that 
$$P=\{\chi_i-\chi_j \mid 1\leq i\neq j \leq l\} \bigcup \{\pm (\chi_i+\chi_j) \mid 1\leq i < j \leq l\} \bigcup \{\pm 2\chi_i \mid 1\leq i \leq l\}. $$
\end{exercise}
\begin{exercise}[Co-characters]
Compute the co-charactrs $Q=P^{\vee} \subset X_*(T)$ and show that $P^{\vee}=\{\mu_i-\mu_j \mid 1\leq i\neq j \leq l\} \bigcup \{\pm (\mu_i+\mu_j) \mid 1\leq i < j \leq l\} \bigcup \{\pm \mu_i \mid 1\leq i \leq l\}$.
\end{exercise}

\begin{exercise}
Consider the isotropic flag 
$${0}\subset V_1 \subset \ldots \subset V_i \subset \ldots \subset V_l $$
where $V_i=<e_1, \ldots, e_i >$.
Then $B$, the stabiliser of this flag, is a Borel subgroup. Let us explicitely compte this subgroup. For $g\in B$,
$g(e_i)\in V_i= <e_1, \ldots, e_i >$ for $1\leq i \leq l$. Now, we know $G(W^{\perp})\subset W^{\perp}$. Thus, 
$$V_l^{\perp} \subset V_{l-1}^{\perp} \subset \ldots \subset V_i^{\perp} \subset \ldots \subset V_1^{\perp} \subset {0}^{\perp}$$
and is naturally fixed by $g$ as well. That is,
$$<e_{1}, \ldots, e_{l}> \subset <e_{1}, \ldots, e_{l}, e_{-l}> \subset \ldots \subset <e_1, \ldots, e_l, e_{-2}, \ldots, e_{-l}> \subset V.$$
Thus we get $g(e_{-1})\in <e_1, \ldots, e_l, e_{-1}, \ldots, e_{-l}> $, $g(e_{-2}) \in <e_1, \ldots, e_l, e_{-2}, \ldots, e_{-l}> $, and, in general $g(e_{-i})\in <e_1, \ldots, e_l, e_{-i}, \ldots, e_{-l}> $. Hence, the elements of $B$ are of the form $\begin{pmatrix} W&X \\ & \tra W^{-1} \end{pmatrix}$ where $W$ is upper triangular.
Thus, a Borel in $Sp_{2l}$ containing the fixed diagonal maximal torus is
$$\left \{\begin{pmatrix} W&X \\ & \tra W^{-1} \end{pmatrix} \mid W \right\}\bigcap Sp_{2l}.$$
\end{exercise}
\begin{exercise}
Thus a system of positive roots is
$$R^{+}=\{\chi_i \pm \chi_j \mid 1\leq i < j \leq l\} \bigcup \{2\chi_i \mid 1\leq i \leq l\} $$
and a set of simple roots is 
$$\Delta=\{\chi_1-\chi_2, \ldots, \chi_{l-1}-\chi_{l}, 2\chi_l \} $$
\end{exercise}

\begin{exercise}[Weyl group]
Show that the Weyl group is $S_l \ltimes (\mathbb Z/2\mathbb Z)^l$.
\end{exercise}

\section{$GSp_{2l}$}
The similitude symplectic group is
 $$GSp_{2l}=\{X\in GL_{2l} \mid \tr XJX = \lambda J, \text{\ for some\ } \lambda\in k^*\}.$$
The scalar $\lambda=\lambda_X$ varies with $X$ and is called similitude factor of the element $X$. 
\begin{exercise}
The map $\Psi \colon GSp_{2l} \rightarrow k^*$ given by $X\mapsto \lambda_X$ is a group homomorphism with kernel $Sp_{2l}$. 
\end{exercise}
 \begin{exercise}
 The set of scalars $\{\alpha. I \in GL_{2l} \mid \alpha\in k^*\}$ is the center of $GSp_{2l}$.
\end{exercise}
\begin{exercise}
The elements of $GSp_{2l}$ induce automorphisms of $Sp_{2l}$ acting via conjugation.
\end{exercise}


\chapter{Orthogonal groups}

Let $k$ be an algebraically closed field of $char\neq 2$. 
Let $V$ be a vector space of dimension $n$. We write $n=2l$ or $2l+1$ depending on when $n$ is even or odd. We fix a basis $\{e_1, \ldots, e_l, e_{-1}, \ldots, e_{-l}\}$ when $n$ is even and $\{e_0, e_1, \ldots, e_l, e_{-1}, \ldots, e_{-l}\}$ when $n$ is odd. Recall that the orthogonal group is $$O_{n} = \{g\in GL_{n} \mid \tra g J g =J\}$$ where $J=\begin{pmatrix} & I_l \\ I_l &\end{pmatrix}$ when $n=2l$ and $J=\begin{pmatrix} 1&&\\& & I_l \\& I_l &\end{pmatrix}$ when $n=2l+1$. Its connected component is the special orthogonal group $SO_n=O_n \cap SL_n$. Its Lie algebra is $$\mathfrak{o}_{n} = \{X\in \mathfrak{gl}_{n} \mid \tra XJ = -JX\}.$$

\section{Root datum of $D_l$ type}
Let us first deal with even dimensional one which is also called $D_l$ type.
\begin{exercise}\label{basis-o2l}
Let $\begin{pmatrix} W&X \\Y & Z \end{pmatrix}$ be in $\mathfrak{o}_{2l}$. Show that $Z=-\tra W$, $X= -\tra X$ and $Y= -\tra Y$. Hence, the following set of elements for $1\leq i, j \leq l$,
\begin{eqnarray*}
(e_{i,j}-e_{-j,-i}) &\text{for}\; i\neq j,\\ 
(e_{i,-j}-e_{j,-i})&\text{for}\; i<j,\\ 
(e_{-i,j}-e_{-j,i})& \text{for}\; i<j,
\end{eqnarray*}
form a basis of the Lie algebra together with the diagonals $e_{i,i}-e_{-i,-i}$.
\end{exercise}

\begin{exercise}
The set $T=\{\diag(\lambda_1, \ldots, \lambda_l, \lambda_1^{-1}, \ldots, \lambda_l^{-1})\mid \lambda_i\in k^*\}$ is a maximal torus in $O_{2l}$. Thus, the rank of $O_{2l}$ is $l$.
\end{exercise}

\begin{exercise}
Fix characters $\chi_i \colon T \rightarrow \mathbb G_m$ by $\chi_i(\diag(\lambda_1, \ldots, \lambda_l, \lambda_1^{-1}, \ldots, \lambda_l^{-1}))= \lambda_i$. Show that $X^*(T)\cong \mathbb Z^l$.
\end{exercise}

\begin{exercise}
Fix co-characters $\mu_i \colon \mathbb G_m \rightarrow T$ by $\chi_i(t) = \diag(1, \ldots, t\ldots ,1 \ldots,  , 1)$ where $t$ is at $i^{th}$ place. Show that $X_*(T)\cong \mathbb Z^l$.
\end{exercise}

\begin{exercise}
Consider the $Ad \colon O_{2l} \rightarrow GL(\mathfrak{o}_{2l})$ given by $Ad(g)(X)=gXg^{-1}$. We restrict this map to the maximal torus $T$ and obtain the simultaneous decomposition of $\mathfrak{o}_{2l}$. Show that the diagonals form a $T$-invariant subspace with respect to $0$ character. Show that each of the basis vector in the above exercise ~\ref{basis-o2l}, other than diagonals, forms a $1$ dimensional $T$-invariant subspace.
\end{exercise}

\begin{exercise}
Compute $P\subset X^*(T)$. Show that 
$$P=\{\chi_i-\chi_j \mid 1\leq i\neq j \leq l\} \bigcup \{\pm (\chi_i+\chi_j) \mid 1\leq i < j \leq l\}. $$
\end{exercise}
\begin{exercise}
Compute $Q=P^{\vee} \subset X_*(T)$. This gives us the root datum $(X^*(T), P, X_*(T), P^{\vee})$.
\end{exercise}

\begin{exercise}[Borel subgroup]
Consider the isotropic flag 
$${0}\subset V_1 \subset \ldots \subset V_i \subset \ldots \subset V_l $$
where $V_i=<e_1, \ldots, e_i >$.
Then $B$, the stabiliser of this flag, is a Borel subgroup. Explicitely compute this and get the elements of $B$  which are of the form $\begin{pmatrix} W&X \\ & \tra W^{-1} \end{pmatrix}$ where $W$ is upper triangular.
\end{exercise}

\begin{exercise}
The Weyl group is isomorphic to $S_l\ltimes (\mathbb Z/2\mathbb Z)^{l-1}$.
\end{exercise}
\section{Root datum of type $B_l$}
Now to deal with $n=2l+1$. 
\begin{exercise}
The Lie algebra elements are  
$$\mathfrak{o}_{2l+1}=\left\{\begin{pmatrix} &\tra c& -\tra b \\ b& W &X \\-c& Y & - \tra W \end{pmatrix}\mid  X= -\tra X, Y= -\tra Y \right\}.$$
Thus along with diagonals, the following form a basis, $1\leq i, j \leq l$,
\begin{eqnarray*}
e_{i,0} -e_{0,-i} \\
e_{0,i} - e_{-i,0}\\
(e_{i,j}-e_{-j,-i}) &\text{for}\; i\neq j,\\ 
(e_{i,-j}-e_{j,-i})&\text{for}\; i<j,\\ 
(e_{-i,j}-e_{-j,i})& \text{for}\; i<j,
\end{eqnarray*}
\end{exercise}

\begin{exercise}
The set $T=\{\diag(1, \lambda_1, \ldots, \lambda_l, \lambda_1^{-1}, \ldots, \lambda_l^{-1})\mid \lambda_i\in k^*\}$ is a maximal torus in $O_{2l+1}$. Thus, the rank of $O_{2l+1}$ is $l$.
\end{exercise}

\begin{exercise}
Compute $P\subset X^*(T)$ by the action of maximal torus on the Lie algebra. Show that 
$$P=\{\pm \chi_i \mid 1\leq i\leq l\} \bigcup \{\chi_i-\chi_j \mid 1\leq i\neq j \leq l\} \bigcup \{\pm (\chi_i+\chi_j) \mid 1\leq i < j \leq l\}. $$
\end{exercise}
This root datum is dual of $C_l$ type.
\begin{exercise}
Compute the co-characters $P^{\vee}$.
\end{exercise}
\begin{exercise}
Compute the centraliser of the element $\diag(1, -1, -1)$ in $SO(3)$.  
\end{exercise}

\chapter{Quaternion Algebra}\label{quaternion}

Quaternion algebras are central object in mathematics. The real quaternions were discovered by Hamilton in 1843. However one can define quaternions over any field $k$. There are several ways to construct a quaternion algebra. We collect  some of them here for fields of characteristic $\neq 2$.

\section{Introduction and Definition}

Let $k$ be a field. In this note we consider algebras over $k$ which are associative (unless mentioned otherwise) with identity and need not be commutative. A $k$-algebra which has no proper two-sided ideal with center $k$ is called a central simple algebra. Artin-Wedderburn theorem gives structure of such a finite dimensional algebra. It says that any central simple algebra is isomorphic to $M_n(D)$ where $D$ is a central division algebra over $k$. A quaternion algebra (sometimes called generalised quaternion algebra)  over $k$ is a central simple algebra over $k$ of dimension four. The Artin-Wedderburn theorem implies that any quaternion algebra is either isomorphic to $M_2(k)$ (called split quaternion algebra) or is isomorphic to a division algebra over $k$ of dimension $4$. In this note we will describe several ways of constructing quaternions over a field of characteristic $\neq 2$.

Apart from being a central object in mathematics, specially in number theory and algebraic groups, quaternions play fundamental role in physics and computer graphics. We will briefly indicate these in the section on applications. 

Let $k$ be a field of characteristic $\neq 2$. 
Let $a,b\in k^*$. Then we can define a multiplication on the four dimensional $k$-vector space $Q=k.1\oplus ki\oplus kj \oplus kij$  by $i^2=a$, $j^2=b$, $ij=-ji$ and $1$ is multiplicative identity. This makes $Q$ a central simple algebra and hence is a quaternion algebra. We denote this algebra $Q=\left(\frac{a,b}{k}\right)$. In fact, any quaternion algebra is of this form. 

We have conjugation operation on $Q$ defined by $\bar x:=x_0-x_1i-x_2j-x_3ij$ where $x=x_0+x_1i+x_2j+x_3ij$. This satisfies the following relation $\overline{xy}=\bar y\bar x$. We can define trace and norm of an element by $tr(x):=x+\bar x =2x_0$ and $N(x):=x\bar x=x_0^2-x_1^2a-x_2^2b+x_3^2ab$ and check that every element $x\in Q$ satisfies a quadratic polynomial $X^2-tr(x)X+N(x)=0$. The norm on $Q$ is a quadratic form $<1,-a,-b, ab>$ which, in fact, is a Pfister form $<1,-a><1,-b>$. Further the norm is multiplicative, i.e., $N(xy)=N(x)N(y)$. We denote the space of pure quaternions by $Q_0=\{x\in Q\mid tr(x)=0\}=\{x\in Q\mid x_0=0\}$. The space $Q_0$ is characterised by $x\in Q$ such that $x^2\in k$ and has quadratic form $N_0=<-a, -b, ab>$. Observe that $N_0$ is restriction of $N$ to $Q_0$. We abuse notation and denote the corresponding bilinear form by $N_0$ as well.

In the constructions below we would need to understand quadratic \'{e}tale algebras. Any such algebra is of the form $K\cong k[X]/<X^2-a>$ for some $a\in k^*$. Thus $K\cong k\times k$ (also called split) or is a quadratic field extension. The algebra $K$ is split if and only if $a\in (k^*)^2$. By taking norm we get norm subgroup $D_a=N_{K/k}(K^*)=\{x-ay^2\}\subset k^*$. To get a quaternion algebra we need $b\in k^*$ as well. If we take $b\in k^*-D_a$ then we get quaternion division algebras.

\section{Graded tensor product}

We refer to~\cite{gr} chapter 8, page 71 for this section. 
Let $A$ and $B$ be two finite dimensional associative (not necessarily commutative) $k$-algebras with identity. Then the algebra $A\otimes B$ is an algebra where multiplication is defined by $(a\otimes b). (a'\otimes b') = aa'\otimes bb'$. Further if $A$ and $B$ both are commutative then $A\otimes B$ is commutative.
 
Now suppose $A=A_0\oplus A_1$ and $B=B_0\oplus B_1$ are $\mathbb Z/2\mathbb Z$ graded where the subscript denotes the degree; $0$ degree term and degree $1$ term. Then we have graded tensor product $A\hat\otimes B$ with algebra product defined as follows:
$$(a\otimes b). (a'\otimes b') = (-1)^{\deg(b)\deg(a')}aa'\otimes bb'$$ where degree $0$ term is $(A_0\otimes B_0) \oplus (A_1\otimes B_1)$ and degree $1$ term is $(A_0\otimes B_1) \oplus (A_1\otimes B_0)$. Notice that even if $A$ and $B$ are commutative $A\hat\otimes B$ need not be. 

We begin with two quadratic \'{e}tale algebras $K$ and $L$ over $k$, say $K=k[\alpha]=k\oplus k\alpha$ and $L=k[\beta]=k\oplus k\beta$ with $\alpha^2=a$ and $\beta^2=b$. We have natural grading on $K$ and $L$. We consider $K\hat\otimes L$ with basis $1=1\otimes 1$, $i=\alpha\otimes 1$, $j=1\otimes \beta$  and $ij=\alpha\otimes \beta$. We claim that this gives us a quaternion algebra $Q\cong\left(\frac{a,b}{k}\right)$. We note that,
\begin{eqnarray*}
i^2&=&(\alpha\otimes 1).(\alpha\otimes 1)=(-1)^{0.1}\alpha^2\otimes 1^2=a.1\\
j^2&=&(1\otimes\beta).(1\otimes\beta)=(-1)^{1.0}1^2\otimes\beta^2=b.1\\
ij&=& (\alpha\otimes 1).(1\otimes\beta)=(-1)^0\alpha\otimes\beta=\alpha\otimes\beta\\
ji&=& (1\otimes\beta).(\alpha\otimes 1)=(-1)^{1.1}\alpha\otimes\beta=-ij.
\end{eqnarray*}
\begin{proposition}
Let $K$ and $L$ be two quadratic \'{e}tale extensions of $k$. Then $K\hat\otimes L$ is a quaternion algebra. Conversely, if $Q$ is a quaternion algebra over $k$ then there exist quadratic \'{e}tale extensions $K$ and $L$ of $k$ such that $Q\cong K\hat\otimes L$.
\end{proposition}
\begin{proof}
Form the discussion above it follows that given $K$ and $L$, the graded tensor product $K\hat\otimes L$ is a quaternion algebra. 
For converse, let $Q=\left(\frac{a,b}{k}\right)$. Then consider $K=k[i]$ and $L=k[j]$. This will give us $Q\cong K\hat\otimes L$.
\end{proof}
Further if either $K$ or $L$ is split then $Q$ is definitely split. In case both $K$ and $L$ are field and either $b\not\in D_a$ or $a\not\in D_b$ then $Q$ is division. 

\section{Composition Algebras and Doubling Method}

A composition algebra $C$ over a field $k$ is a not necessirly associative algebra over $k$ with identity together with a non-degenerate quadratic form $N$ on $C$ which is multiplicative, i.e., $N(xy) = N(x)N(y)$ for all $x,y \in C$.
It turns out that the possible dimensions of a composition algebra over $k$ are $1, 2, 4, 8$. 
Composition algebras of dimension $1$ and $2$ are commutative and associative, those of dimension $4$ are associative but not commutative and are quaternion algebras, and those of dimension $8$ are neither commutative nor associative called octonion algebra or  Cayley algebra. For more details we refer to the chapter 1 of \cite{sv}.

Let $K$ be a quadratic \'{e}tale algebra over $k$. Denote the non-trivial $k$-automorphism of $K$ by $x\mapsto \overline x$. Let $\lambda\in k^*$ and $Q=K\oplus K$. Define multiplication on $Q$ as follows:
$$(x_1,y_1).(x_2,y_2)=(x_1x_2+\lambda y_1\overline y_2, x_1y_2+\overline x_2y_1).$$
This makes $Q$ an algebra and is called doubling method. 
Suppose $K=k[\alpha]$ with $\alpha^2=a\in k^*$. The automorphism of $K$ is $x=x_1+\alpha y_1\mapsto \overline x=x_1-\alpha y_1$, i.e., $\overline \alpha=-\alpha$. Clearly $(1,0)$ is identity of $Q$. Denote $i=(\alpha,0)$ and $j=(0,1)$. Then
\begin{eqnarray*}
i^2&=& (\alpha,0)(\alpha,0)= (\alpha^2, 0)= a(1, 0)\\
j^2 &=& (0,1)(0,1)=(\lambda .1, 0)=\lambda (1, 0) \\
ij&=& (\alpha, 0)(0,1) = (0, \alpha)=\alpha(0,1)\\
ji&=& (0,1)(\alpha, 0) = (0,\overline\alpha)=-\alpha(0,1)=-ij
\end{eqnarray*}
Thus $Q\cong\left(\frac{a,b}{k}\right)$ where $a=\alpha^2$ and $b=\lambda$.
Further any quaternion algebra can be constructed this way. Clearly if $K$ is split or $\lambda\in D_a$ we get split quaternion algebra.


\section{Clifford Algebra}

Let $V$ be a vector space of dimension $n$ over $k$. Suppose $V$ has a quadratic form $q$ on it. Then we have Clifford algebra $C(V,q)$ of dimension $2^n$ defined by:
$$C(V,q)=\frac{T(V)}{\langle\{v\otimes v-q(v).1\mid v\in V\}\rangle}$$
where $T(V)$ is tensor algebra.
This algebra is $\mathbb Z/2\mathbb Z$ graded where grading is induced from that of $T(V)$. We refer to~\cite{gr}, chapter 8 for more details.

Let $V$ be a two dimensional vector space with a non-degenerate quadratic form $q$ on it. The Clifford algebra construction above gives us an algebra of dimension $4$. We claim that this algebra is a quaternion algebra. Suppose $\{x,y\}$ is an orthogonal basis of $V$ with $q(x)=a$ and $q(y)=b$ both $a,b\in k^*$ as $q$ is non-degenerate. Then $T(V)\cong k[x,y]$, non-commutative polynomial ring in two variables, and $C(V,q)\cong k[x,y]/\langle x^2-a, y^2-b, xy+yx\rangle$. By denoting $i=x$ and $j=y$ we see that $C(V,q)\cong \left(\frac{a,b}{k}\right)$. Further any quaternion algebra can be obtained this way by taking $V$ as space generated by orthogonal basis $i$ and $j$.

What condition on $(V,q)$ gives a division quaternion algebra?

\section{Cyclic Algebra}

Let $K/k$ be a cyclic field extension of degree $n$, i.e., $K/k$ is Galois with Galois group $Gal(K/k)\cong \mathbb Z/n\mathbb Z=<\sigma>$. Take $b\in k^*$. We define multiplication on $n$-dimensional $K$-vector space $A=K.1\oplus Ke\oplus Ke^2\oplus\cdots\oplus Ke^{n-1}$ to make it an algebra as follows:
$$e^n=b, \ e.\beta = \sigma(\beta) e.$$ 
The algebra $A=(K/k, \sigma, b)$ is called a cyclic algebra. For further details we refer to~\cite{pi} chapter 15, in particular, section 15.1.

We begin with a quadratic \'{e}tale algebra $K=k[\alpha]\cong k[X]/<X^2-a>$ over $k$ with $\sigma$ the non-trivial automorphism of $K$. We fix $b\in k^*$. Consider a two-dimensional $K$-vector space $Q=K.1\oplus K.e$ and define multiplication by $e^2=b$ and $e.\beta =\sigma(\beta)e$. We claim that the algebra obtained this way is a quaternion algebra. Define $i=\alpha 1$ and $j=1e$. Then,
$i^2=\alpha^2=a$, $j^2=1e.1e= e^2=b$, $ij=\alpha 1.1e=\alpha e$, $ji=1e.\alpha 1 = \sigma(\alpha)e=-\alpha e=-ij$.
Further any quaternion algebra can be obtained this way. Note that if $K$ is a field and $b\not\in D_a$ we get a division algebra.

\section{3-dimensional quadratic space with trivial discriminant}

To motivate the construction of this section we begin with analysing the multiplication on quaternions. We have $Q= k\oplus Q_0$. Thus multiplication on $Q$ 
\begin{eqnarray*}
xy&=&(x_0+x_1i+x_2j+x_3ij)(y_0+y_1i+y_2j+y_3ij)\\
&=&(x_0y_0+x_1y_1a+x_2y_2b-x_3y_3ab ) + (x_0y_1+x_1y_0-x_2y_3b+x_3y_2b)i \\
&& (x_0y_2+x_1y_3a+x_2y_0-x_3y_1a)j + (x_0y_3+x_1y_2-x_2y_1 + x_3y_0)ij
\end{eqnarray*}
can be rewritten as: 
$$xy=(x_0y_0-N_0(x,y), x_0w+y_0v+v\times w)$$
where $x=(x_0,v)$ and $y=(y_0,w)$ for $v,w\in Q_0$ and $\times\colon Q_0\times Q_0\rightarrow Q_0$ is a cross product defined by $v\times w=-(x_2y_3-x_3y_2)bi+(x_1y_3-x_3y_1)aj+(x_1y_2-x_2y_1)ij$. We need to check that this is a unique well defined cross product (i.e., $(Q_0,\times)$ satisfies properties of being a Lie algebra) on $Q_0$ with respect to $N_0$. Other properties are easy to verify using the above formula except the Jacobi identity: 
$u\times(v\times w)+w\times (u\times v)+v\times (w\times u)=0.$
We verify the following formal formula to compute the cross product: 
\begin{eqnarray*} v\times w &= & \det\begin{pmatrix}-bi&-aj&ij\\ x_1& x_2& x_3\\ y_1&y_2 &y_3\end{pmatrix}\\
u\times(v\times w) &= & \det\begin{pmatrix}-bi&-aj&ij\\ u_1& u_2& u_3\\ -b(x_2y_3-x_3y_2) & a(x_1y_3-x_3y_1) &x_1y_2-x_2y_1\end{pmatrix} \\ 
&=& vN_0(u,w)-wN_0(u,v).
\end{eqnarray*}
Hence,
\begin{eqnarray*}
 && u\times(v\times w)+w\times (u\times v)+v\times (w\times u)\\
&=& vN_0(u,w)-wN_0(u,v) + uN_0(w,v) - vN_0(w,u) + wN_0(v,u) - uN_0(v,w)\\
&=&0
\end{eqnarray*} 

We see that $(Q_0,N_0)$ has trivial discriminant which determines the cross product. Check that for any $u=u_1i+u_2j+u_3ij$ we get $N_0(u,v\times w)=u_1(x_2y_3-x_3y_2)ba-u_2(x_1y_3-x_3y_1)ab +u_3(x_1y_2-x_2y_1)ab=ab\det\begin{pmatrix}u_1 &u_2&u_3\\x_1& x_2& x_3\\y_1&y_2 &y_3\end{pmatrix}$. Note that $\bigwedge^3V=<i\wedge j\wedge ij>\cong k$ and $\bigwedge^3N_0(i\wedge j\wedge ij)=(-a)(-b)ab=(ab)^2$. Hence the map $\psi\colon \bigwedge^3(Q_0,N_0)\rightarrow (k,<1>)$ given by $i\wedge j\wedge ij\mapsto ab.1$ (i.e., multiplication by $ab$) is an isomorphism of one-dimensional quadratic spaces where $<1>$ denotes the bilinear form $(x,y)\mapsto xy$ on $k$. Thus $v\times w$ is uniquely determined by $N_0(u,v\times w)=\psi(u\wedge v\wedge w)$.

We give a construction for quaternion algebra starting from a rank $3$ quadratic space $V$ over $k$, with trivial discriminant. This construction is adaptation of the construction for Cayley algebras in \cite{th}.  
Let $B\colon V\cross V\rightarrow k$ be a non-degenerate bilinear form. Suppose that the discriminant of $(V,B)$ is trivial. That is, there exists $\psi\colon \bigwedge^3(V,B)\rightarrow (k,<1>)$ which is an isomorphism of quadratic spaces. We use this to define a vector product $\cross\colon V\cross V\longrightarrow V$ by the formula, $B(u,v\cross w)=\psi(u\wedge v\wedge w)$, for $u,v,w\in V$. 
Let $Q= k\oplus V$ be a $k$-vector space of dimension $4$. 
We define a multiplication on $Q$ by,
$$
(a,v)(b,w)=(ab-B(v,w), aw+bv+v\cross w)
$$
where $a,b\in k$ and $v,w\in V$. 
With this multiplication, $Q$ is a quaternion algebra over $k$, with norm $N(a,v)=a^2+B(v,v)$. 
The isomorphism class of $Q$ is independent of $\psi$ chosen. 
One can show that all quaternion algebras arise this way. 

\subsection{$3$-dimensional vector space and cross product}

Let $V$ be a $3$ dimensional vector space with a non-degenerate quadratic form $N=\langle a,b,c\rangle$ with respect to basis $\{e_1,e_2,e_3\}$. We wish to define a cross product on $V$ uniquely determined by $N$.

Denote $\tilde N=\langle \frac{c}{b}, \frac{c}{a},ab\rangle $. Now define the cross product $\times \colon V\times V \rightarrow V$ by the following formal formula: 
$$v\times w = \det \begin{pmatrix} b e_1 & ae_2 & \frac{c}{ab}e_3 \\ x_1 &x_2 & x_3 \\ y_1 & y_2 & y_3 \end{pmatrix}$$
where $v=(x_1,x_2,x_3)$ and $w=(y_1,y_2,y_3)$. 
It is easy to verify that this is bilinear and $v\times v=0$. To verify the Jacobi identity we first check the following formula $u\times(v\times w) = v\tilde N(u,w) - w\tilde N(u,v)$ for all $u,v,w\in V$. 
\begin{eqnarray*}
u\times (v\times w) &=& \det\begin{pmatrix} b e_1 & ae_2 & \frac{c}{ab}e_3 \\ u_1 & u_2 & u_3 \\ b(x_2y_3-x_3y_2) & -a(x_1y_3 - x_3y_1) & \frac{c}{ab}(x_1y_2-x_2y_1) \end{pmatrix} \\
&=& x_1e_1(\frac{c}{a}u_2y_2 + ab u_3y_3) - y_1e_1 (\frac{c}{a}u_2x_2 + ab u_3x_3) + x_2e_2(\frac{c}{b}u_1y_1 + ab u_3y_3) \\&& - y_2e_2 (\frac{c}{b}u_1x_1 + ab u_3x_3) +
x_3e_3(\frac{c}{b}u_1y_1 + \frac{c}{a} u_2y_2) - y_3e_3 (\frac{c}{b}u_1x_1 + \frac{c}{a} u_2x_2)   \\
&=& v\tilde N(u,w) - w\tilde N(u,v)
\end{eqnarray*}

Further we want the cross product to be uniquely determined by $N$, i.e., $N(u, v\times w)=0$ if $u=v$ or $u=w$. We verify the following formal formula:
$$N(u,v\times w)=\det \begin{pmatrix}abu_1 & ab u_2 & \frac{c^2}{ab}u_3 \\ x_1 & x_2 & x_3 \\ y_1 & y_2 & y_3\end{pmatrix}$$
and the equation $N(v,v\times w)=0$ gives $c^2=(ab)^2$, i.e., $Q$ should have trivial discriminant. 

\section{Geometric Method} 
To a quaternion algebra $Q=\left(\frac{a,b}{k}\right)$ one associates a conic $C(a,b)$ given by the equation $ax^2+by^2=z^2$ in the projective space $P^2$.  The attached conic does not depend on choice of a basis and isomorphism of $Q$ is equivalent to $k$-isomorphism of associated conics. Further one can prove that the quaternion algebra $Q$ is split if and only if the attached conic $C$ has a $k$-rational point. We refer the section 1.3 in \cite{gz} for this construction and proofs.

\section{Embedding inside Matrix algebra}

Let $Q=\left(\frac{a,b}{k}\right)$ be a quaternion algebra. Denote $K=k[X]/<X^2-a>=k[\alpha]$. The following map gives an embedding of $Q$ inside $M_2(K)$: 
$$(x_0+x_1i)+(x_2+x_3i)j=z_1+z_2j\mapsto \begin{pmatrix}z_1 &bz_2\\ \overline{z_2} & \overline{z_1}\end{pmatrix}.$$
One can also check that norm and trace is determinant and trace of the corresponding matrix.


\section{Some Facts and Applications}

\subsection{Classification over certain fields}
Classifying quaternion algebras over a field up to isomorphism is not an easy task. Here we mention some for example.
\begin{theorem}
\begin{enumerate}
\item Over an algebraically closed field the only quaternion is matrix algebra. 
\item Over a finite field the only quaternion algebra is matrix algebra.
\item Over reals any quaternion algebra is either isomorphic to $M_2(\mathbb R)$ or $\mathbb H=\left(\frac{-1,-1}{\mathbb R}\right)$.
\end{enumerate}
\end{theorem}

For the airthmetic aspect of quaternions the standard source of reference is~\cite{vi}.

\subsection{Groups of type $A_1$}

Let $G$ be an algebraic group defined over $k$. It is one of the important problem in the theory of algebraic groups to classify them. Given a quaternion algebra $Q$ over $k$ one can form the group $SL_1(Q)$ which is a group defined over $k$ of type $A_1$. In fact any group of type $A_1$ can be obtained this way.

\subsection{Merkurjev-Suslin Theorem}
\begin{theorem}
Any $2$-torsion element in Brauer group of $k$ is equivalent to $Q_1\otimes Q_2\otimes\cdots \otimes Q_n$ where $Q_i$ are quaternion algebra over $k$.
\end{theorem}





\begin{thebibliography}{99}
\normalsize
\bibitem[Bo]{bo} Borel, Armand; {\it ``Linear algebraic groups''}, Second edition. Graduate Texts in Mathematics, 126. Springer-Verlag, New York, 1991. xii+288 pp.
\bibitem[BG]{bg} Benson, C. T.; Grove, L. C.; {\it ``Finite reflection groups''}, Bogden \& Quigley, Inc., Publishers, Tarrytown-on-Hudson, N.Y., 110 pp. 1971. 
\bibitem[Ca]{ca} Carter, Roger W.; {\it ``Simple groups of Lie type''}, Pure and Applied Mathematics, Vol. 28. John Wiley \& Sons, London-New York-Sydney, 1972. viii+331 pp. 
\bibitem[Ca2]{ca2} Carter, Roger W.; {\it ``Finite groups of Lie type - Conjugacy classes and complex characters''}, Pure and Applied Mathematics (New York). A Wiley-Interscience Publication. John Wiley \& Sons, Inc., New York, 1985. xii+544 pp.
\bibitem[Ca3]{ca3} Carter, R. W.; {\it ``Lie algebras of finite and affine type''}, Cambridge Studies in Advanced Mathematics, 96. Cambridge University Press, Cambridge, 2005. xviii+632 pp.
\bibitem[CS]{cs} Conway, John H.; Smith, Derek A.; {\it``On quaternions and octonions: their geometry, arithmetic, and symmetry''}, A K Peters, Ltd., Natick, MA, 2003.
\bibitem[CSM]{csm} Carter, Roger; Segal, Graeme; Macdonald, Ian; {\it ``Lectures on Lie groups and Lie algebras - With a foreword by Martin Taylor''}, London Mathematical Society Student Texts, 32. Cambridge University Press, Cambridge, 1995. vii+190 pp.
\bibitem[Dg]{Dg} de Graaf, Willem Adriaan; {\it ``Computation with linear algebraic groups''}, Monographs and Research Notes in Mathematics. CRC Press, Boca Raton, FL, 2017. xiv+327 pp. ISBN: 978-1-4987-2290-2. 
\bibitem[EW]{ew} Erdmann, Karin; Wildon, Mark J.; {\it ``Introduction to Lie algebras''}, Springer Undergraduate Mathematics Series. Springer-Verlag London, Ltd., London, 2006. x+251 pp.
\bibitem[Ge]{ge} Geck, Meinolf; {\it ``An introduction to algebraic geometry and algebraic groups''}, First paperback reprinting of the 2003 original. Oxford Graduate Texts in Mathematics, 20. Oxford University Press, Oxford, 2013. xii+307 pp.
\bibitem[Gr]{gr} Grove, Larry C.; {\it ``Classical groups and geometric algebra''}, Graduate Studies in Mathematics, 39. American Mathematical Society, Providence, RI, 2002.
\bibitem[GZ]{gz} Gille, Philippe; Szamuely, Tamás; {\it ``Central simple algebras and Galois cohomology''}, Cambridge Studies in Advanced Mathematics, 101. Cambridge University Press, Cambridge, 2006. 
\bibitem[Hu]{hu} Humphreys, James E.; {\it ``Linear algebraic groups''}, Graduate Texts in Mathematics, No. 21. Springer-Verlag, New York-Heidelberg, 1975. xiv+247 pp.
\bibitem[Hu2]{hu2} Humphreys, James E.; {\it ``Introduction to Lie algebras and representation theory''}, Graduate Texts in Mathematics, Vol. 9. Springer-Verlag, New York-Berlin, 1972. xii+169 pp. 
\bibitem[Hu3]{hu3} Humphreys, James E.; {\it ``Reflection groups and Coxeter groups''}, Cambridge Studies in Advanced Mathematics, 29. Cambridge University Press, Cambridge, 1990. xii+204 pp. 
\bibitem[La]{la} Lam, T. Y.; {\it ``Introduction to quadratic forms over fields''}, Graduate Studies in Mathematics, 67. American Mathematical Society, Providence, RI, 2005.
\bibitem[LT]{lt} Lehrer, Gustav I.; Taylor, Donald E.; {\it ``Unitary reflection groups''}, Australian Mathematical Society Lecture Series, 20. Cambridge University Press, Cambridge, 2009. viii+294 pp.
\bibitem[MT]{mt} Malle, Gunter; Testerman, Donna; {\it ``Linear algebraic groups and finite groups of Lie type''}, Cambridge Studies in Advanced Mathematics, 133. Cambridge University Press, Cambridge, 2011. 
\bibitem[Pi]{pi} Pierce, Richard S. ;{\it ``Associative algebras''}, Graduate Texts in Mathematics, 88. Studies in the History of Modern Science, 9. Springer-Verlag, New York-Berlin, 1982.
\bibitem[Sp]{sp} Springer, T. A.; {\it ``Linear algebraic groups''}, Reprint of the 1998 second edition. Modern Birkhäuser Classics. Birkhäuser Boston, Inc., Boston, MA, 2009. xvi+334 pp. 
\bibitem[St]{st} Steinberg, Robert; {\it ``Conjugacy classes in algebraic groups''}, Notes by Vinay V. Deodhar. Lecture Notes in Mathematics, Vol. 366. Springer-Verlag, Berlin-New York, 1974. vi+159 pp. 
\bibitem[SV]{sv} Springer, Tonny A.; Veldkamp, Ferdinand D., {\it`` Octonions, Jordan algebras and exceptional groups''}, Springer Monographs in Mathematics. Springer-Verlag, Berlin, 2000.
\bibitem[Ta]{ta} Taylor, Donald E.; {\it ``The geometry of the classical groups''}, Sigma Series in Pure Mathematics, 9. Heldermann Verlag, Berlin, 1992. xii+229 pp.
\bibitem[Th]{th} Thakur, Maneesh, {\it`` Cayley algebra bundles on $A^2_K$ revisited''}, Comm. Algebra 23 (1995), no. 13, 5119-5130.
\bibitem[TY]{ty} Tauvel, Patrice; Yu, Rupert W. T.; {\it ``Lie algebras and algebraic groups''}, Springer Monographs in Mathematics. Springer-Verlag, Berlin, 2005. xvi+653 pp.
\bibitem[Vi]{vi}  Vignéras, Marie-France; {\it `` Arithmétique des algèbres de quaternions. (French) [Arithmetic of quaternion algebras]''}, Lecture Notes in Mathematics, 800. Springer, Berlin, 1980.
\end{thebibliography}
\end{document}